\documentclass[reqno,a4paper]{amsart}

\usepackage{amsthm}
\usepackage{amsmath}
\usepackage[all]{xy} \SelectTips{eu}{}
\usepackage{latexsym,amsfonts,amssymb}
\usepackage{hyperref}
\hypersetup{colorlinks=true, linkcolor=teal}
\usepackage{cite}
\hypersetup{citecolor=brown}
\usepackage{mathptmx}
\usepackage{nicefrac}
\usepackage{array}
\usepackage[dvipsnames]{xcolor}
\usepackage{mathrsfs}
\usepackage{rotating}
\usepackage{bbm}
\usepackage{tikz-cd}

\renewcommand{\theequation}{$\smash{\sharp}\mspace{0.5mu}$\arabic{equation}}

\usepackage{vmargin}
\setpapersize{A4}
\setmargrb{25mm}{20mm}{25mm}{20mm}

\newcommand{\numberseries}{\mdseries}   %Fontseries used for numbering theorem

\newlength{\thmtopspace}                %Space above theorem
\newlength{\thmbotspace}                %Space below theorem
\newlength{\thmheadspace}               %Space between theorem caption and text
\newlength{\thmindent}                  %For indenting

\newtheoremstyle{bfupright head,slanted body}
                {\thmtopspace}{\thmbotspace}
                {\slshape}{\thmindent}{\bfseries}{.}{\thmheadspace}
                {{\numberseries \thmnumber{{\bf #2} }}\thmnote{#3}}

\newtheoremstyle{bfupright head,upright body}
                {\thmtopspace}{\thmbotspace}
                {\upshape}{\thmindent}{\bfseries}{.}{\thmheadspace}
                {{\numberseries \thmnumber{{\bf #2} }}\thmnote{#3}}

\newtheoremstyle{bfit head,upright body}
                {\thmtopspace}{\thmbotspace}
                {\upshape}{\thmindent}{\upshape}{.}{\thmheadspace}
                {{\numberseries\thmnumber{{\bf #2} }}
                {\bfseries\itshape\thmnote{\negthickspace#3}}}

\newtheoremstyle{it head,upright body}
                {\thmtopspace}{\thmbotspace}
                {\upshape}{\thmindent}{\upshape}{.}{\thmheadspace}
                {{\numberseries\thmnumber{{\bf #2} }}
                {\itshape\thmnote{\negthickspace#3}}}

\newtheoremstyle{fixed bf head,slanted body}
                {\thmtopspace}{\thmbotspace}{\slshape}
                {\thmindent}{\bfseries}{\!.}{\thmheadspace}
                {{\thmname{#1}\numberseries\thmnumber{ {\bf #2}}}\thmnote{ (#3)} }

\newtheoremstyle{fixed bf head,upright body}
                {\thmtopspace}{\thmbotspace}{\upshape}
                {\thmindent}{\bfseries}{\!.}{\thmheadspace}
                {{\thmname{#1}\numberseries\thmnumber{ {\bf #2}}}\thmnote{ (#3)} }

\newtheoremstyle{independent paragraph}
                {\thmtopspace}{\thmbotspace}
                {\upshape}{\thmindent}{\upshape}{}{0pt}
                {\thmnote{#3 }}

\newtheoremstyle{subparagraph}
                {\thmbotspace}{\thmbotspace}
                {\upshape}{\thmindent}{\upshape}{}{0pt}
                {\thmnote{#3 }}

\newtheoremstyle{notes}
                {\thmtopspace}{\thmbotspace}
                {\ttfamily}{\thmindent}{\ttfamily\small }{}{0pt}
                {\thmnote{#3 }}
                
\newtheoremstyle{numbered paragraph}
                {\thmtopspace}{\thmbotspace}{\upshape}
                {\thmindent}{\upshape}{}{\thmheadspace}
                {{\numberseries \thmnumber{\bf #2.}}}

\theoremstyle{bfupright head,slanted body}
\newtheorem{res}{}[section]             \newtheorem*{res*}{}

\theoremstyle{bfit head,upright body}
                 \newtheorem*{com*}{}

\theoremstyle{bfupright head,upright body}
               \newtheorem*{bfhpg*}{}

\theoremstyle{it head,upright body}
               \newtheorem*{ithpg*}{}

\theoremstyle{fixed bf head,slanted body}
\newtheorem{thm}[res]{Theorem}          \newtheorem*{thm*}{Theorem}
\newtheorem{prp}[res]{Proposition}      \newtheorem*{prp*}{Proposition}
\newtheorem{cor}[res]{Corollary}        \newtheorem*{cor*}{Corollary}
\newtheorem{lem}[res]{Lemma}            \newtheorem*{lem*}{Lemma}

\theoremstyle{fixed bf head,upright body}
\newtheorem{dfn}[res]{Definition}       \newtheorem*{dfn*}{Definition}
      \newtheorem*{obs*}{Observation}
\newtheorem{rmk}[res]{Remark}           \newtheorem*{rmk*}{Remark}
\newtheorem{exa}[res]{Example}          \newtheorem*{exa*}{Example}
         \newtheorem*{exe*}{Exercise}
            \newtheorem{stp*}{Setup}
         \newtheorem{ntn*}{Notation}
     \newtheorem{con*}{Construction}

      \newtheorem{telconj*}{Conjecture}

\theoremstyle{numbered paragraph}

\theoremstyle{subparagraph}

\theoremstyle{notes}

\newlength{\thmlistleft}        %leftmargin
\newlength{\thmlistright}       %rightmargin
\newlength{\thmlistpartopsep}   %partopsep
\newlength{\thmlisttopsep}      %topsep
\newlength{\thmlistparsep}      %parsep
\newlength{\thmlistitemsep}     %itemsep

\newcounter{eqc} 
\newenvironment{eqc}{\begin{list}{\upshape (\textit{\roman{eqc}})}%
    {\usecounter{eqc}%
      \setlength{\leftmargin}{\thmlistleft}%
      \setlength{\labelwidth}{\thmlistleft}%
      \setlength{\rightmargin}{\thmlistright}%
      \setlength{\partopsep}{\thmlistpartopsep}%
      \setlength{\topsep}{\thmlisttopsep}%
      \setlength{\parsep}{\thmlistparsep}%
      \setlength{\itemsep}{\thmlistitemsep}}}%
  {\end{list}}%

\newcommand{\eqclbl}[1]{{\upshape(\textit{#1})}}

\newcounter{prt}
\newenvironment{prt}{\begin{list}{\upshape (\alph{prt})}%
    {\usecounter{prt}%
      \setlength{\leftmargin}{\thmlistleft}%
      \setlength{\labelwidth}{\thmlistleft}%
      \setlength{\rightmargin}{\thmlistright}%
      \setlength{\partopsep}{\thmlistpartopsep}%
      \setlength{\topsep}{\thmlisttopsep}%
      \setlength{\parsep}{\thmlistparsep}%
      \setlength{\itemsep}{\thmlistitemsep}}}%
  {\end{list}}%

\newcommand{\prtlbl}[1]{{\upshape(#1)}}

\newcounter{rqm}
\newenvironment{rqm}{\begin{list}{\upshape (\arabic{rqm})}%
    {\usecounter{rqm}%
      \setlength{\leftmargin}{\thmlistleft}%
      \setlength{\labelwidth}{\thmlistleft}%
      \setlength{\rightmargin}{\thmlistright}%
      \setlength{\partopsep}{\thmlistpartopsep}%
      \setlength{\topsep}{\thmlisttopsep}%
      \setlength{\parsep}{\thmlistparsep}%
      \setlength{\itemsep}{\thmlistitemsep}}}%
  {\end{list}}%

\newcommand{\rqmlbl}[1]{{\upshape(#1)}}

\newcounter{rqmm}
  {\end{list}}%

  {\end{list}}%

  {\end{list}}%

\newcommand{\pgref}[1]{\ref{#1}}

\renewcommand{\eqref}[1]{(\pgref{eq:#1})}

\newcommand{\corref}[2][Corollary ]{#1\pgref{cor:#2}}
\newcommand{\dfnref}[2][Definition~]{#1\pgref{dfn:#2}}

\newcommand{\lemref}[2][Lemma ]{#1\pgref{lem:#2}}

\newcommand{\prpref}[2][Proposition ]{#1\pgref{prp:#2}}

\newcommand{\thmref}[2][Theorem ]{#1\pgref{thm:#2}}

\makeatletter
\def\@nobreak@#1{\mathchoice%
  {\nobreakdef@\displaystyle\f@size{#1}}%
  {\nobreakdef@\nobreakstyle\tf@size{\firstchoice@false #1}}%
  {\nobreakdef@\nobreakstyle\sf@size{\firstchoice@false #1}}%
  {\nobreakdef@\nobreakstyle\ssf@size{\firstchoice@false #1}}%
  \check@mathfonts}%
\def\nobreakdef@#1#2#3{\hbox{{%
                    \everymath{#1}%
                    \let\f@size#2\selectfont%
                    #3}}}%
\makeatother

\DeclareFontFamily{T1}{cmex}{}
\DeclareFontShape{T1}{cmex}{m}{n}{<-> s * [0.89] cmex10}{}
\DeclareSymbolFont{cmlargesymbols}{T1}{cmex}{m}{n}
\DeclareMathSymbol{\mycoprod}{\mathop}{cmlargesymbols}{"60} 
\DeclareMathSymbol{\myprod}{\mathop}{cmlargesymbols}{"51} \let\prod\myprod

\DeclareSymbolFont{usualmathcal}{OMS}{cmsy}{m}{n}
\DeclareSymbolFontAlphabet{\mathcal}{usualmathcal}

\DeclareSymbolFont{letters}{OML}{txmi}{m}{it}

\DeclareMathSymbol{\alpha}{\mathord}{letters}{"0B}
\DeclareMathSymbol{\beta}{\mathord}{letters}{"0C}
\DeclareMathSymbol{\gamma}{\mathord}{letters}{"0D}
\DeclareMathSymbol{\sigma}{\mathord}{letters}{"0E}
\DeclareMathSymbol{\epsilon}{\mathord}{letters}{"0F}
\DeclareMathSymbol{\zeta}{\mathord}{letters}{"10}
\DeclareMathSymbol{\eta}{\mathord}{letters}{"11}
\DeclareMathSymbol{\theta}{\mathord}{letters}{"12}
\DeclareMathSymbol{\iota}{\mathord}{letters}{"13}
\DeclareMathSymbol{\kappa}{\mathord}{letters}{"14}
\DeclareMathSymbol{\lambda}{\mathord}{letters}{"15}
\DeclareMathSymbol{\mu}{\mathord}{letters}{"16}
\DeclareMathSymbol{\nu}{\mathord}{letters}{"17}
\DeclareMathSymbol{\xi}{\mathord}{letters}{"18}
\DeclareMathSymbol{\pi}{\mathord}{letters}{"19}
\DeclareMathSymbol{\rho}{\mathord}{letters}{"1A}
\DeclareMathSymbol{\sigma}{\mathord}{letters}{"1B}
\DeclareMathSymbol{\tau}{\mathord}{letters}{"1C}
\DeclareMathSymbol{\upsilon}{\mathord}{letters}{"1D}
\DeclareMathSymbol{\phi}{\mathord}{letters}{"1E}
\DeclareMathSymbol{\chi}{\mathord}{letters}{"1F}
\DeclareMathSymbol{\psi}{\mathord}{letters}{"20}
\DeclareMathSymbol{\omega}{\mathord}{letters}{"21}
\DeclareMathSymbol{\varepsilon}{\mathord}{letters}{"22}
\DeclareMathSymbol{\vartheta}{\mathord}{letters}{"23}
\DeclareMathSymbol{\varpi}{\mathord}{letters}{"24}
\DeclareMathSymbol{\varrho}{\mathord}{letters}{"25}
\DeclareMathSymbol{\varsigma}{\mathord}{letters}{"26}
\DeclareMathSymbol{\varphi}{\mathord}{letters}{"27}

\DeclareMathSymbol{\Theta}{\mathord}{letters}{"02}
\DeclareMathSymbol{\Xi}{\mathord}{letters}{"04}
\DeclareMathSymbol{\Pi}{\mathord}{letters}{"05}
\DeclareMathSymbol{\Upsilon}{\mathord}{letters}{"07}
\DeclareMathSymbol{\Phi}{\mathord}{letters}{"08}
\DeclareMathSymbol{\Psi}{\mathord}{letters}{"09}
\DeclareMathSymbol{\upGamma}{\mathalpha}{operators}{"00}
\DeclareMathSymbol{\upDelta}{\mathalpha}{operators}{"01}
\DeclareMathSymbol{\upTheta}{\mathalpha}{operators}{"02}
\DeclareMathSymbol{\upLambda}{\mathalpha}{operators}{"03}
\DeclareMathSymbol{\upXi}{\mathalpha}{operators}{"04}
\DeclareMathSymbol{\upPi}{\mathalpha}{operators}{"05}
\DeclareMathSymbol{\upSigma}{\mathalpha}{operators}{"06}
\DeclareMathSymbol{\upUpsilon}{\mathalpha}{operators}{"07}
\DeclareMathSymbol{\upPhi}{\mathalpha}{operators}{"08}
\DeclareMathSymbol{\upPsi}{\mathalpha}{operators}{"09}
\DeclareMathSymbol{\upOmega}{\mathalpha}{operators}{"0A}

\DeclareMathAlphabet\PazoBB{U}{fplmbb}{m}{n}%

\newcommand{\Hom}[3]{\operatorname{Hom}_{#1}(#2,#3)}
\newcommand{\Ext}[4]{\operatorname{Ext}_{#1}^{#2}(#3,#4)}

\renewcommand{\Im}[1]{\operatorname{Im}\mspace{1mu}#1}
\newcommand{\Ker}[1]{\operatorname{Ker}\mspace{1mu}#1}

\newcommand{\lMod}[1]{{}_{#1}\mspace{-1mu}\operatorname{Mod}}

\newcommand{\lPrj}[1]{{}_{#1}\mspace{-1mu}\operatorname{Prj}}
\newcommand{\lInj}[1]{{}_{#1}\mspace{-1mu}\operatorname{Inj}}

\newcommand{\QSD}[2]{\mathbf{D}_{#1}(#2)}

\newcommand{\alg}{A}

\newcommand{\Cq}[1]{C_{\mspace{-1mu}\smash{#1}}}
\newcommand{\Sq}[1]{S_{\mspace{-4mu}\smash{#1}}}
\newcommand{\Kq}[1]{K_{\smash{#1}}}
\newcommand{\Fq}[1]{F_{\mspace{-5mu}\smash{#1}}}

\newcommand{\cH}[2][1]{\mathbb{H}^{#1}_{\textnormal{\tiny[}#2\textnormal{\tiny]}}}

\newcommand{\lE}[1]{{}_{#1}\mathscr{E}}

\newcommand{\supp}[1]{\operatorname{supp}\mspace{1mu}#1}

\DeclareMathOperator{\Inj}{Inj}
\newcommand{\Prj}[1]{\operatorname{Prj}\mspace{2mu}#1}

\DeclareMathOperator{\susp}{\operatorname{\Sigma}}
\DeclareMathOperator{\syz}{\operatorname{\Omega}}
\DeclareMathOperator{\stmod}{\underline{mod}}
\DeclareMathOperator{\StMod}{\underline{Mod}}
\DeclareMathOperator{\fpmod}{mod}
\DeclareMathOperator{\Mod}{Mod}

\DeclareMathOperator{\htpy}{Ho}

\begin{document}

\title{Silting $\text{t}$-structures in $Q$-shaped derived categories}

\author{Anastasios Slaftsos}

\address{Anastasios Slaftsos, Dipartimento di Matematica ``Tullio Levi-Civita'', Universit\'a degli Studi di Padova, via Trieste 63, 35131 Padova, Italy}
\email{slaftsos@math.unipd.it}



%16G20 Representations of quivers and partially ordered sets
%16E35 Derived categories and associative algebras
%18G25 Relative homological algebra, projective classes 
%18G80 Derived categories, triangulated categories
%18N40 Homotopical algebra, Quillen model categories, derivators

\thanks{\textbf{Acknowledgements:} The author would like to thank Henrik Holm and Jorge Vit\'oria for helpful discussions on this work. 
The author was supported by NextGenerationEU, Mission 4 Component 1 Investment 4.1~, CUP C96E23000600001. }
  
\begin{abstract}
    Torsion pairs, and in particular t-structures, play a central role in the study of triangulated categories. Specifically, t-structures induced by silting (or tilting) objects often admit desirable properties with strong connections to derived equivalences. In this paper, using the correspondence of Saor\'{i}n-{\v{S}}{\v{t}}ov\'{i}{\v{c}}ek between cohereditary cotorsion pairs in Frobenius exact categories and t-structures in their stable categories, we construct a family of t-structures in the $Q$-shaped derived category of Holm and J\o rgensen, arising from admissible partitions of $Q$. We give an explicit description of the associated cotorsion pairs inside the Frobenius exact category of the bifibrant objects, and we identify the corresponding co-aisles by certain homological vanishing conditions. Such t-structures are proved to be induced by a silting object, that can be completely determined by the combinatorics of $Q$.  Finally, we illustrate our results by recovering well-known equivalences in the $Q$-shaped setting, while also providing examples where the combinatorial conditions fail (e.g. cyclic quivers), showing that such categories may admit no non-trivial t-structures, revealing phenomena analogous to those observed by Linckelmann in stable module categories.
\end{abstract}

\maketitle

\section{Introduction}

Holm and J\o rgensen in \cite{HJ-JLMS} introduced $Q$-shaped derived categories as a generalisation of derived categories, effectively providing a homological algebra setup in which (co)chain complexes are replaced by $Q$-shaped analogues for a suitable small $k$-preadditive category $Q$, satisfying certain conditions (see \cite{MR4786505} for a brief survey).  Such categories are compactly generated (cf. \cite{HJ-TAMS}) algebraic triangulated categories and in particular they can be realised as stable categories of the Frobenius exact category of bifibrant objects.\smallbreak

For certain examples of $Q$, it is well-known that the $Q$-shaped derived category is equivalent to the (usual) derived category of some ring. Such instances of triangle equivalences can be found in \cite{MR3742439} for the derived category of $N$-complexes, in \cite{saito2023tilting} for $m$-periodic derived categories and more generally, in \cite{gratz2024tilting} for a class of categories $Q$, arising from graded finite dimensional self-injective algebras. These results were later generalised in the framework of dg-categories by Jasso in \cite{jasso2025qshaped}.  Consequently, a natural question is whether every $Q$-shaped derived category is equivalent to a derived category of some ring. In this paper we provide a negative answer to this question.\smallbreak

Torsion pairs in abelian and triangulated categories, and in particular t-structures \cite{BeBeDe}, play a central role in the study of triangulated categories. Silting and cosilting objects generalise tilting theory and admit intrinsic formulations in terms of t-structures in abstract triangulated settings \cite{psaroudakis2018realisation, nicolas2019silting}. Hearts induced by silting t-structures often enjoy strong categorical properties and provide a bridge between triangulated and abelian methods. In many situations, these hearts retain enough information to reconstruct or compare triangulated structures (for example when the object is tilting). This highlights their importance in representation theory, where silting theory controls mutations \cite{aihara_iyama_silting,angeleri2025mutation}, derived equivalences \cite{Happel_morita_theory_fda,JRi89b,Keller_Deriving_DG_categories}, and others. \smallbreak

In this paper, we study t-structures in $Q$-shaped derived categories. In particular, we construct a family of (cohereditary and complete) cotorsion pairs in the category of $\lMod{A}$-valued representations of $Q$, that descend to t-structures in the $Q$-shaped derived category, using Saor{\'{\i}}n-{\v{S}}t'ov{\'\i}{\v{c}}ek's correspondence theorem \cite{SaorinStovicek}. For what follows, denote by ${}^{\perp_1}\mathscr{E}$ the Frobenius exact category consisting of cofibrant objects, whose stable category is equivalent to the $Q$-shaped derived category $\QSD{Q}{A}$ of a ring $A$. The following result gives an explicit description of such t-structures in the stable category ${}^{\perp_1}\underline{\mathscr{E}}$ which is equivalent to the $Q$-shaped derived category of $A$.

\begin{res*}[Theorem~A]
  Let $(\mathcal{L},\mathcal{R})$ be a ``suitably nice'' partition of $Q$. Consider the cotorsion pair $\,\mathfrak{C}=(\mathcal{X}_\mathcal{L},\mathcal{Y}_\mathcal{L})$ generated by the set $\{S_q(A)\,|\, q\in \mathcal{L}\}$ of stalk objects supported in $\mathcal{L}$,  in the Frobenius exact category ${}^{\perp_1}\mathscr{E}$. The image $(\mathfrak{X}_\mathcal{L},\mathfrak{Y}_\mathcal{L})$ of the cotorsion pair in the $Q$-shaped derived category $\QSD{Q}{A}$ is given by
\begin{align*}
  \phantom{XXX}
  \mathfrak{X}_\mathcal{L}
  &\,=\,
    \left\{ X \in \QSD{Q}{A} \;\left|\mspace{-5mu} 
    \begin{array}{l}
    \text{There exists an object $P$ with $\supp{P} \subseteq \mathcal{L}$ and $P(q) \in \lPrj{A}$} \\
    \text{for every $q \in \mathcal{L}$ and an isomorphism $P \cong X$ in $\QSD{Q}{A}$}
    \end{array}
    \mspace{-8mu}
    \right.
    \right\} ,
  \\  
  \mathfrak{Y}_\mathcal{L}
  &\,=\,
  \{ Y \in \QSD{Q}{A} \ | \ \cH[1]{q}(Y)=0 \text{ for every } q \in \mathcal{L} \} \;.  
\end{align*}
The cotorsion pair $\,\mathfrak{C}$ corresponds to a t-structure $\,\mathbf{t}\,=\, (\Omega\mathfrak{X}_\mathcal{L},\mathfrak{Y}_\mathcal{L})$ in $\QSD{Q}{A}$.
\end{res*}

Such t-structures can be proved to be induced by a \emph{silting object}. The partition $(\mathcal{L},\mathcal{R})$ of $Q$ allows one to define the boundary $\partial\mathcal{L}$ of $\mathcal{L}$ as the set of all vertices in $\mathcal{L}$ that map non-trivially into $\mathcal{R}$. 
Representations over a field $\Bbbk$ of the sub-quiver $\partial\mathcal{L}$ can be identified as modules over a finite dimensional algebra $\Lambda$, which we call the \emph{boundary algebra} of $Q$. In particular, one can show that $\lMod{\Lambda}$ is homologically embedded in $\QSD{Q}{\Bbbk}$, see Lemma \ref{lem: homological embedding}.  

\begin{res*}[Theorem~B]
   The functor $\psi\colon\lMod{\Lambda}\longrightarrow\QSD{Q}{\Bbbk}$ maps tilting $\Lambda$-modules to silting objects in $\QSD{Q}{\Bbbk}$. Moreover, the t-structures induced by the images of projective generators of $\Lambda$ coincide with $ (\mathfrak{X}_\mathcal{L},\Sigma\,\mathfrak{Y}_\mathcal{L})$ in $\QSD{Q}{\Bbbk}$.
\end{res*}

This result can be extended to the $Q$-shaped derived category of any $\Bbbk$-algebra $A$ by an application of \cite[Cor.~3.31]{jasso2025qshaped}. In this context, one recovers some already known equivalences shown in \cite{MR3742439,gratz2024tilting}, and provides a sufficient combinatorial condition for the existence of silting objects in such categories. We recover these equivalences using our combinatorial approach in Section \ref{sec:examples}. Finally, for examples where the combinatorial conditions fail (e.g. for cyclic quivers) we show that such categories may admit no non-trivial t-structures, revealing phenomena analogous to those observed by Linckelmann \cite{linkelmann} in stable module categories.\smallbreak

\subsection*{Notation and conventions}
\label{notation}

Unless otherwise stated, all subcategories are strict and full. If $\mathcal{A}$ is an additive category admitting direct limits, we say that an object $X$ of $\mathcal{A}$ is \textbf{finitely presented} if $\Hom{\mathcal{A}}{X}{-}$ commutes with direct limits. For a set of objects $\mathcal{X}$ of an additive category $\mathcal{A}$, we denote by $\operatorname{add}\mathcal{X}$ (respectively, $\operatorname{Add}\mathcal{X}$) the subcategory of $\mathcal{A}$ whose objects are summands of finite (respectively, set-indexed) coproducts of objects in $\mathcal{X}$.  For any class $\mathcal{L}$ of objects in an abelian category $\mathcal{A}$, and denoting by $\Ext{\mathcal{A}}{1}{-}{-}$ the pairing associating to two objects the Yoneda extensions between them, we consider the following subcategories
\begin{align*}
  {}^{\perp_1}\mathcal{L}
  &\;=\;
  \{X \in \mathcal{A} \ |\, \Ext{\mathcal{A}}{1}{X}{L}=0 \text{ for every } L \in \mathcal{L} \}
  \\
  \mathcal{L}^{\perp_1}
  &\;=\;
  \{Y \in \mathcal{A} \ |\, \Ext{\mathcal{A}}{1}{L}{Y}=0 \text{ for every } L \in \mathcal{L} \}\;.
\end{align*}

For $\mathcal{T}$ a triangulated category, $M$ an object in $\mathcal{T}$ and $I$ a subset of the integers $\mathbb{Z}$ we denote perpendicular classes as follows:
\begin{align*}
{}^{\perp_I}M\,&=\,\{Y\in\mathcal{T}\,|\, \Hom{\mathcal{T}}{Y}{\Sigma^iM}=0,\text{ for all } i\in I\}\\
M^{\perp_I}\,&=\,\{Y\in\mathcal{T}\,|\, \Hom{\mathcal{T}}{M}{\Sigma^iY}=0,\text{ for all } i\in I\}.  
\end{align*}\smallbreak

The word \textit{module} will stand for \textit{left module}. For a ring $A$, we denote by $\lMod{A}$ the category of left $A$-modules. For a left $A$-module $M$ (or more generally, for an object $M$ in an abelian category $\mathcal{A}$) we denote by $\operatorname{pd}M$ its projective dimension and by  $\operatorname{id}M$ its injective dimension. Unless otherwise stated, $k$ will denote a commutative ring.

\subsection*{Structure of the paper}
 We begin in Section \ref{sec:preliminaries} with preliminaries concerning cotorsion pairs in Frobenius exact categories and their interplay with torsion pairs in the associated stable category, via a Theorem of Saor{\'{\i}}n-{\v{S}}t'ov{\'\i}{\v{c}}ek \cite{SaorinStovicek}. Finally, we give a brief overview of the construction of the $Q$-shaped homotopy category of Holm and J\o rgensen, focusing on its Frobenius model.\smallbreak

 In the second Section \ref{sec:t-structures} we define admissible partitions of the small category $Q$. This allows us to define complete and cohereditary cotorsion pairs in the Frobenius exact category ${}^{\perp_1}\mathscr{E}$ of the bifibrant objects in $\lMod{Q,A}$. In such cotorsion pairs, the left orthogonal class is characterised in terms of the support of the representations in the admissible partition, whereas the right orthogonal class admits a cohomological description. We conclude the section by showing that these cotorsion pairs descend to t-structures in the stable category which coincides with the $Q$-shaped derived category of $A$.\smallbreak

 In Section \ref{sec:tilting} we work over an algebraically closed field $\Bbbk$, and based on the partition of $Q$ we define a finite dimensional $\Bbbk$-algebra, called the \emph{boundary algebra}. We show that its module category homologically embeds in the category $\lMod{Q}$, and that the composition of this embedding with the canonical localisation functor $\lMod{Q}\to \QSD{Q}{\Bbbk}$, maps tilting $\Lambda$-modules to silting objects in $\QSD{Q}{\Bbbk}$. Finally, we prove that the t-structures induced by the images of the small  projective generators of $\Lambda$ coincide with the t-structure obtained in Section \ref{sec:t-structures}.  \smallbreak

 Finally, in Section \ref{sec:examples} we utilise our results to recover well-known equivalences in the $Q$-shaped world, shown in \cite{MR3742439,gratz2024tilting}. We also provide non-examples, where the combinatorial conditions fail (e.g. for cyclic quivers) we show that such categories do not admit any non-trivial t-structures and by extension no non-trivial silting objects.

\section{Preliminaries}\label{sec:preliminaries}

In what follows, we give a concise introduction to the notions of cotorsion pairs in an exact category and the concept of t-structures in a triangulated category. Furthermore, we give a brief exposition of the interplay between these two notions in the context of Frobenius categories and their stable category, as shown by Saor{\'{\i}}n and {\v{S}}t'ov{\'\i}{\v{c}}ek in \cite{SaorinStovicek}. Finally, we recall the basic theory and construction of the so-called $Q$-shaped derived category of a ring, a notion introduced by Holm and J{\o}rgensen in \cite{HJ-JLMS} and the main category of our study.

\subsection{Cotorsion pairs in exact categories}
%\phantom{.} \vspace*{1.5ex}

A class $\mathcal{L}$ of objects in an exact category $\mathscr{C}$ is \textbf{generating} if for every $X \in \mathscr{C}$ there exists a deflation $L \twoheadrightarrow X$ with $L \in \mathcal{L}$. Dually, $\mathcal{L}$ is \textbf{cogenerating} if for every $X \in \mathscr{C}$ there exists an inflation $X \rightarrowtail L$ with $L \in \mathcal{L}$. Let $\Prj{\mathscr{C}}$ be the class of projective objects and $\Inj{\mathscr{C}}$ the class of injective objects in $\mathscr{C}$, see \cite[Definition~11.1 and Rmk.~11.2]{Buhler}. One says that $\mathscr{C}$ \textbf{has enough projectives} if $\Prj{\mathscr{C}}$ is generating; dually, $\mathscr{C}$ \textbf{has enough injectives} if $\Inj{\mathscr{C}}$ is cogenerating. See \cite[Definition~11.9]{Buhler}.\smallbreak

A \textbf{cotorsion pair} in $\mathscr{C}$ is a pair $(\mathcal{A},\mathcal{B})$ of classes of objects in $\mathscr{C}$ satisfying $\mathcal{A}^{\perp_1} = \mathcal{B}$ and $\mathcal{A} = {}^{\perp_1}\mathcal{B}$. To every class $\mathcal{L}$ of objects in $\mathscr{C}$ there are two associated cotorsion pairs,
\begin{equation*}
  \mathfrak{G}_\mathcal{L}
  \;=\;
  ({}^{\perp_1}(\mathcal{L}^{\perp_1}),\mathcal{L}^{\perp_1})
  \qquad \text{and} \qquad
  \mathfrak{C}_\mathcal{L}
  \;=\;
  ({}^{\perp_1}\mathcal{L},({}^{\perp_1}\mathcal{L})^{\perp_1})\;.
\end{equation*}

We call $\mathfrak{G}_\mathcal{L}$ the cotorsion pair \textbf{generated} by $\mathcal{L}$ and $\mathfrak{C}_\mathcal{L}$ the cotorsion pair \textbf{cogenerated}\footnote{\,This terminology is the one used by G{\"o}bel and Trlifaj \cite[Definition~2.2.1]{GobelTrlifaj} and Saor{\'{\i}}n and {\v{S}}t'ov{\'\i}{\v{c}}ek \cite[\S1.3]{SaorinStovicek}. Beware that some authors, like Enochs and Jenda \cite[Definition~7.1.2]{rha}, use the ``opposite'' terminology.}~by~$\mathcal{L}$. Note that, for any class $\mathcal{L}$ of objects in $\mathscr{C}$, there are always inclusions $\Prj{\mathscr{C}} \subseteq {}^{\perp_1}\mathcal{L}$  and $\operatorname{Inj}{\mathscr{C}} \subseteq \mathcal{L}^{\perp_1}$.\smallbreak
A cotorsion pair $(\mathcal{A},\mathcal{B})$ in $\mathscr{C}$ is said to be \textbf{complete} if the following conditions hold:
\begin{rqm}
\item For every $X \in \mathscr{C}$ there exists a conflation $B \rightarrowtail A \twoheadrightarrow X$ in $\mathscr{C}$ with $A \in \mathcal{A}$ and $B \in \mathcal{B}$.  In this case we say that the cotorsion pair has enough projectives, or equivalently that the class $\mathcal{A}$ is \textbf{special precovering}.
\item For every $X \in \mathscr{C}$ there exists a conflation $X \rightarrowtail B' \twoheadrightarrow A'$ in $\mathscr{C}$ with $A' \in \mathcal{A}$ and $B' \in \mathcal{B}$.  In this case we say that the cotorsion pair has enough injectives, or equivalently that the class $\mathcal{B}$ is \textbf{special preenveloping}.
\end{rqm}

Salce's trick \cite{Salce} shows that if $\mathcal{A}$ is generating and $\mathcal{B}$ is cogenerating, then conditions \rqmlbl{1} and \rqmlbl{2} above are equivalent. This follows by inspecting the proofs of \cite[Proposition~7.1.7]{rha} or \cite[Lemma~2.2.6]{GobelTrlifaj} (for module categories) and \cite[Thm.~2.13(5)]{SaorinStovicek} (for exact categories). Lastly, in a Frobenius exact category $\mathscr{C}$ we call a cotorsion pair $(\mathcal{A},\mathcal{B})$ \textbf{hereditary} if $\Omega\,\mathcal{A}\subseteq\mathcal{A}$ and \textbf{cohereditary} if $\Omega^{-1}\mathcal{A}\subseteq\mathcal{A}$, where $\Omega$ is the syzygy functor.\smallbreak

\subsection{t-structures}

    Assume that $\mathcal{C}$ is a subcategory of an additive category $\mathscr{A}$, we denote\footnote{This is an abbreviation of notation. Whenever the orthogonal is denoted without any further decoration, eg. ${}^\perp\mathcal{C}$ we will always mean the Hom-orthogonal, i.e. ${}^{\perp_0}\mathcal{C}$} by ${}^\perp\mathcal{C}$ the subcategory of all objects $X$ in $\mathscr{A}$ with $\Hom{\mathscr{A}}{X}{C}=0$ for all $C$ in $\mathcal{C}$. Dually, one defines $\mathcal{C}^\perp$. Similarly, one can extend the definition of Hom-orthogonality in a triangulated category $\mathscr{T}$.
    
\begin{dfn}
    Let $\mathcal{T}$ be a triangulated category and denote by $\Sigma$ the suspension functor. A pair of subcategories $(\mathcal{X},\mathcal{Y})$ in $\mathcal{T}$ is called a \textbf{torsion pair}, if $\mathcal{X},\mathcal{Y}$ are both closed under direct summands, they are a Hom-orthogonal pair, that is, $\Hom{\mathcal{T}}{X}{Y}=0$ for every $X$ in $\mathcal{X}$ and $Y$ in $\mathcal{Y}$ and every object $Z$ in $\mathcal{T}$ fits into a triangle 
    \begin{equation}\label{eq: approximation triangle}
    X\longrightarrow Z\longrightarrow Y\longrightarrow \Sigma X
        \end{equation}
    with $X$ in $\mathcal{X}$ and $Y$ in $\mathcal{Y}$. Furthermore, a torsion pair\footnote{Note that in our definition of the t-structure the two subcategories form a torsion pair. One should be cautious as plenty of authors in the literature define a t-structure $(\mathcal{X},\mathcal{Y})$ in such a way that $(\mathcal{X},\Sigma^{-1}\mathcal{Y})$ is a torsion pair.} $(\mathcal{X},\mathcal{Y})$ in $\mathcal{T}$ is called a \textbf{t-structure} if moreover, $\mathcal{X}$ is closed under positive suspensions. We call $\mathcal{X}$ the \textbf{aisle} and $\mathcal{Y}$ the \textbf{co-aisle} of the t-structure.
\end{dfn}

In \cite{BeBeDe} it is proven that the \textbf{heart} $\mathcal{H}:=\Sigma^{-1}\mathcal{X}\cap\mathcal{Y}$ of the t-structure  $(\mathcal{X},\mathcal{Y})$ is an abelian category whose exact structure is given by the triangles of $\mathcal{T}$ with terms in $\mathcal{H}$. \smallbreak

Given that $\mathcal{T}$ is a triangulated category with arbitrary coproducts, we call an object $C$ in $\mathcal{T}$ \textbf{compact} if the functor $\Hom{\mathcal{T}}{C}{-}$ commutes with coproducts. A t-structure $(\mathcal{X},\mathcal{Y})$ in $\mathcal{T}$ is called \textbf{compactly generated} if $\mathcal{Y}=\mathcal{S}^\perp$ for a set $\mathcal{S}$ of compact objects in $\mathcal{T}$. The subcategory of compact objects in $\mathcal{T}$ is denoted by $\mathcal{T}^c$. Finally, given a subset $\mathcal{S}$ of objects in $\mathcal{T}$, we define by $\operatorname{Loc}(\mathcal{S})$ the smallest triangulated subcategory of $\mathcal{T}$ that contains $\mathcal{S}$ and it is closed under forming arbitrary coproducts. We say that $\mathcal{S}$ \textbf{generates} $\mathcal{T}$, if $\operatorname{Loc}(\mathcal{S})=\mathcal{T}$.

\begin{exa}
    The stereotypical example of a t-structure is the \textbf{standard t-structure} in the derived category of a ring. Let $\mathcal{T}=\mathbf{D}(R)$ be the (unbounded) derived category of a ring $R$. The pair of subcategories $(\mathbf{D}^{\le0},\mathbf{D}^{\ge1})$ in $\mathbf{D}(R)$, given by
    \[
    \begin{aligned}
        \mathbf{D}^{\le0}&:=\{X\in\mathbf{D}(R)\,|\, H^i(X)=0,\text{ for all }i>0\}\\
        \mathbf{D}^{\ge1}&:=\{X\in\mathbf{D}(R)\,|\, H^i(X)=0,\text{ for all }i<1\},
    \end{aligned}
    \]
   is a t-structure in $\mathbf{D}(R)$ whose heart is proven to be $\mathcal{H}=\mathbf{D}^{\le0}\cap\mathbf{D}^{\ge0}\cong \Mod R$. Notice that the standard t-structure in the derived category of a ring always exists. 
\end{exa}

From this point on, assume that $\mathscr{C}$ is a \textbf{Frobenius exact category} and denote with $\underline{\mathscr{C}}$ the associated stable category. This category is triangulated with suspension functor $\susp=\syz^{-1}$ (see \cite{Happel}) which assigns an object $X$ to its co-syzygy $\syz^{-1}(X)$.  Consider the full subcategory $\mathscr{X}\subseteq\mathscr{C}$. We denote its image in the stable category $\underline{\mathscr{C}}$ under the canonical functor $q\colon \mathscr{C}\longrightarrow\underline{\mathscr{C}}$, by $\underline{\mathscr{X}}$. The following theorem provides an important connection between cotorsion pairs of the Frobenius exact category and t-structures in the stable category.

\begin{thm}[\!{\cite[Prop.~3.3~and~Prop.~3.9]{SaorinStovicek}}]\label{thm: soarin-stovicek correspondence}
There is a bijective correspondence between cotorsion pairs $(\mathcal{A},\mathcal{B})$ in $\mathscr{C}$ and Hom-orthogonal pairs $(\mathscr{X},\mathscr{Y})=(\underline{\Omega\mathcal{A}},\underline{\mathcal{B}})$ in $\underline{\mathscr{C}}$. 
Furthermore, it extends to a bijective correspondence between complete cohereditary cotorsion pairs and t-structures. 
\end{thm}

\subsection{The \texorpdfstring{$Q$}{Q}-shaped derived category}
%\phantom{.} \vspace*{1.5ex}
The $Q$-shaped derived category of a ring was introduced in \cite{HJ-JLMS} and generalises the notion of the usual derived category of a ring. We begin by providing a brief overview of the main ideas of the construction of such categories. For more information we refer to \cite{HJ-JLMS,HJ-TAMS,HJ-Abel}.\smallbreak

Consider a $k$-linear category $Q$ (over a commutative ring $k$) which is also Hom-finite, locally bounded, has a Serre functor, admits a nilpotent ideal called the \textbf{pseudoradical} ideal denoted by $\mathfrak{r}$, that in certain examples identifies with the radical ideal of the category $Q$ in the sense of Kelly \cite{Kelly64}, and satisfies the strong retraction property (cf. \cite[Setup~2.1]{HJ-Abel}). Holm and J{\o}rgensen constructed two different combinatorial model structures on the category of \textbf{$Q$-shaped modules} $\lMod{Q,A}=\operatorname{Fun}(Q,\lMod{A})$, which can also be considered as $k$-linear functors $Q\longrightarrow\lMod{A}.$ In particular, they defined the set
\[
\mathscr{E}=\{X\in\lMod{Q,A}\,|\, j_*X\in{}_Q\mathcal{L}\}
\]
where $\lMod{Q}$ is the category of the $Q$-shaped representations with values in $\lMod{k}$, the forgetful functor is denoted by $j_*\colon\lMod{Q,A}\longrightarrow\lMod{Q}$ and $${}_Q\mathcal{L}=\{X\in\lMod{Q}\,|\, \operatorname{pd}_QX<\infty\}=\{X\in\lMod{Q}\,|\, \operatorname{id}_QX<\infty\}.$$ They prove in loc. cit. that the pairs $({}^{\perp_1}\mathscr{E},\mathscr{E})$ and $(\mathscr{E},\mathscr{E}^{\perp_1})$ are complete hereditary cotorsion pairs, both generated by sets (c.f. \cite[Thms.~4.4, 5.5 and 5.9]{HJ-JLMS}).\smallbreak

Consequently, there exist two abelian model structures (in the sense of Hovey \cite{Hovey02}) on $\lMod{Q,A}$. The projective model structure is given by the Hovey triple $(\mathcal{C}_p,\mathcal{W}_p,\mathcal{F}_p)=({}^{\perp_1}\mathscr{E},\mathscr{E},\lMod{Q,A})$, where the cofibrant objects are in ${}^{\perp_1}\mathscr{E}$, the class of trivial objects is $\mathcal{E}$ and everything is fibrant. The injective model structure is given by the Hovey triple $(\mathcal{C}_i,\mathcal{W}_i,\mathcal{F}_i)=(\lMod{Q,A}, \mathscr{E},\mathscr{E}^{\perp_1})$, where the fibrant objects are in $\mathscr{E}^{\perp_1}$, the class of trivial objects is $\mathscr{E}$ and everything is cofibrant. Both of these model structures are combinatorial and as we mention later, they are also stable. Moreover, it is evident that the two model structures share the same trivial objects (and same weak equivalences) therefore, their associated homotopy categories are the same.\smallbreak

Holm and J{\o}rgensen denoted the homotopy category $\htpy(\lMod{Q,A})$ by $\mathbf{D}_Q(A)$, the so-called \textbf{$Q$-shaped derived category} of $A$. In other words, one has
\[
\mathbf{D}_Q(A)=\htpy(\lMod{Q,A})=\lMod{Q,A}\left[\{\text{weq}\}^{-1}\right]
\]

Furthermore, they showed that $\mathbf{D}_Q(A)$ is an algebraic triangulated category in the sense of \cite{Happel}, as it is possible to be defined as the stable category of the Frobenius exact category ${}^{\perp_1}\mathscr{E}$ (resp. $\mathscr{E}^{\perp_1}$) whose projective-injective objects are the categorical projective (resp. injective) objects $\lPrj{Q,A}$ (resp. $\lInj{Q,A}$) of $\lMod{Q,A}$ (see \cite[Thm.~6.5]{HJ-JLMS}). More precisely, one has
\[
\frac{{}^{\perp_1}\mathscr{E}}{\lPrj{Q,A}}\stackrel{\sim}{\longrightarrow}\mathbf{D}_Q(A)\stackrel{\sim}{\longleftarrow}\frac{\mathscr{E}^{\perp_1}}{\lInj{Q,A}}.
\]

It is worth noting that the Q-shaped derived category is a compactly generated triangulated category, as shown in \cite{HJ-TAMS}.\smallbreak

\section{\texorpdfstring{$t$}{t}-structures induced by admissible partitions of \texorpdfstring{$Q$}{Q}}\label{sec:t-structures}

In what follows we utilise the combinatorial structure of the underlying graph of certain $Q$'s in order to define well behaved t-structures in the $Q$-shaped derived category. We use the same setup and notation as in Holm and J{\o}rgensen \cite[Stp.~2.1]{HJ-Abel}. Recall also from \emph{loc. cit.} that the local boundedness of the category $Q$ implies that the following sets are finite
\[
\mathsf{N}_{-}(q)\,=\,\{p\in Q\,|\, Q(p,q)\neq0\}\quad\text{and}\quad \mathsf{N}_{+}(q)\,=\,\{r\in Q\,|\,Q(q,r)\neq0\}.
\]

\begin{dfn}
\label{dfn:L}
Let $\mathcal{L},\mathcal{R}$ be classes of objects of $Q$ that consist a partition of $Q$, such that $\mathcal{L}$ is subject to the following conditions:
\begin{rqm}
\item[(L$_1$)] For every $q \in \mathcal{L}$ one has $\mathsf{N}_-(q) \subseteq \mathcal{L}$; that is, if $q$ is an object in $\mathcal{L}$, then every object $p \in Q$ for which $Q(p,q) \neq 0$ also belongs to $\mathcal{L}$.
\item[(L$_2$)] There does \emph{not} exist an infinite sequence $q_1 \to q_2 \to q_3 \to \cdots$ of non-zero morphisms in the pseudo-radical $\mathfrak{r}$ with $q_1,q_2,q_3,\ldots \in \mathcal{L}$.
\end{rqm}
Dually, $\mathcal{R}$ is subject to the dual condition:
\begin{rqm}
\item[(R$_1$)] For every $q \in \mathcal{R}$ one has $\mathsf{N}_+(q) \subseteq \mathcal{R}$; that is, if $q$ is an object in $\mathcal{R}$, then every object $p \in Q$ for which $Q(q,p) \neq 0$ also belongs to $\mathcal{R}$.
\item[(R$_2$)] There does \emph{not} exist an infinite sequence $ \cdots\to q_3 \to q_2 \to q_1$ of non-zero morphisms in the pseudo-radical $\mathfrak{r}$ with $q_1,q_2,q_3,\ldots \in \mathcal{R}$.
\end{rqm}
We call such a partition of $Q$ an \textbf{admissible partition}.
\end{dfn}

\noindent \textbf{Blanket Assumption.} In the following, we assume that $Q$ is acyclic, in the sense of \cite[Dfn.~5.13]{HJ-TAMS} (in \emph{loc. cit.} they use the terminology ``$Q$ has no cycles''), and $\mathcal{L}$, $\mathcal{R}$ are classes as in \dfnref{L} above.\smallbreak

Recall from \cite{HJ-TAMS} that the \textbf{support} of an object $X \in \lMod{Q,A}$ is given by:
\begin{equation*}
  \supp{X} \,=\, \{q \in Q \,|\, X(q) \neq 0\}\;.
\end{equation*}
Recall from \cite[Prop.~7.15]{HJ-JLMS} the functors $\Cq{q}$, $\Sq{q}$, and $\Kq{q}$ where $q \in Q$ is an object. As shown in \cite[Prop.~7.18]{HJ-JLMS} there are for $X \in \lMod{Q,\,\alg}$ isomorphisms,
\begin{equation*}
  \textstyle
  \Kq{q}(X) \cong\, \bigcap_{g \mspace{1mu}\in\mspace{1mu} \mathfrak{r}(q,*)} \Ker{X(g)}
\quad \text{and} \quad  
\Cq{q}(X) \cong\, X(q)/\sum_{f \mspace{1mu}\in\mspace{1mu} \mathfrak{r}(*,q)} \Im{X(f)} \;,
\end{equation*}
where the intersection is taken over all morphisms $g \in \mathfrak{r}$ with domain $q$, and the sum is taken over all morphisms $f \in \mathfrak{r}$ with codomain $q$. In particular, $\Kq{q}(X)$ can be seen as a submodule of $X(q)$, and $\Cq{q}(X)$ can be seen as a quotient of $X(q)$.

\begin{lem}
  \label{lem:KY}
  Let $Y \neq 0$ be an object in $\lMod{Q,A}$. If $\supp{Y} \subseteq \mathcal{L}$, then there exists $q \in \supp{Y}$ with $\Kq{q}(Y)=Y(q)$, and hence $\Sq{q}(Y(q))$ is a subobject of $Y$ in $\lMod{Q,A}$. 
\end{lem}

\begin{proof}
  Suppose for contradiction that $\Kq{q}(Y) \neq Y(q)$ holds for every $q \in \supp{Y}$. Since $Y$ is non-zero, it has non-empty support, so we can choose some object $q_1 \in \supp{Y}$. By assumption we have $\Kq{q_1}(Y) \neq Y(q_1)$; this implies that there exists some $g_1 \colon q_1 \to q_2$ in $\mathfrak{r}$ with $Y(g_1) \neq 0$; in particular, $g_1 \neq 0$ and $q_2 \in \supp{Y}$. By assumption, $\Kq{q_2}(Y) \neq Y(q_2)$, so there exists a $g_2 \colon q_2 \to q_3$ in $\mathfrak{r}$ with $Y(g_2) \neq 0$; in particular, $g_2 \neq 0$ and $q_3 \in \supp{Y}$. Continuing in this manner, we construct an infinite sequence, 
\begin{equation*}
  \xymatrix@C=1.5pc{  
    q_1 \ar[r]^-{g_1} &
    q_2 \ar[r]^-{g_2} &
    q_3 \ar[r]^-{g_3} &
    \ \cdots 
  },
\end{equation*}
of non-zero morphisms in $\mathfrak{r}$ where $q_1,q_2,q_3,\ldots$ belong to $\supp{Y}$. As $\supp{Y} \subseteq \mathcal{L}$ holds, this contradicts condition (2) in \dfnref{L}. Hence there exists some $q \in \supp{Y}$ with $\Kq{q}(Y)=Y(q)$. By \cite[Prop.~7.15]{HJ-JLMS} there is an adjunction $(\Sq{q},\Kq{q})$; write
$\varepsilon_q$ for the counit. The morphism 
\begin{equation*}
  \xymatrix{
    \Sq{q}(Y(q)) \,=\, \Sq{q}\Kq{q}(Y) \ar[r]^-{\varepsilon_q^Y} & Y
  }
\end{equation*}
in $\lMod{Q,A}$ is monic as $\varepsilon_q^Y(q) \colon Y(q) \to Y(q)$ is the identity map and $\varepsilon_q^Y(p) \colon 0 \to Y(p)$ is the zero map for $p \neq q$. Hence $\Sq{q}(Y(q))$ is a subobject of $Y$ in $\lMod{Q,A}$.
\end{proof}

\begin{prp}
  \label{prp:filtration}
  Every $X \in \lMod{Q,\,A}$ with $\supp{X} \subseteq \mathcal{L}$ admits a filtration, that is, a sequence of monomorphisms
  \begin{equation}
  \label{eq:trans-filt}
   0 = X_0 \rightarrowtail X_1 \rightarrowtail X_2 \rightarrowtail \cdots  \rightarrowtail X_\omega \rightarrowtail X_{\omega+1} \rightarrowtail \cdots 
  \end{equation}
  with the following properties:
  \begin{rqm}
  \item $X_\lambda = X$ for some ordinal $\lambda$.
  \item For every limit ordinal $\beta \le \lambda$ one has \smash{$X_\beta = \varinjlim_{\alpha<\beta} X_\alpha$}.
  \item For every successor ordinal $\alpha+1 \le \lambda$ one has $X_{\alpha+1}/X_\alpha = \Sq{q_\alpha}(X(q_\alpha))$ for some object $q_\alpha \in \supp{X} \smallsetminus \supp{X_\alpha}$.
  \end{rqm}  
\end{prp}

\begin{proof}
  We start by constructing a filtration \eqref{trans-filt} that satisfies conditions \rqmlbl{1}, \rqmlbl{2}, and 
  \begin{eqc}
    \item[\rqmlbl{3$'$}] For every successor ordinal $\alpha+1 \le \lambda$ one has $X_{\alpha+1}/X_\alpha = \Sq{q_\alpha}(X(q_\alpha)/X_\alpha(q_\alpha))$ for some object $q_\alpha \in \supp{X/X_\alpha}$.
  \end{eqc}  
Set $X_0=0$. Let $\beta$ be an ordinal and assume that $X_\alpha$ has been constructed for every $\alpha<\beta$. If $\beta$ is a limit ordinal, define $X_\beta$ as in \rqmlbl{2}. Now assume that $\beta = \alpha+1$ is a successor ordinal. If $X_\alpha = X$, then \rqmlbl{1} holds with $\lambda = \alpha$ and we are done; otherwise, $Y_\alpha = X/X_\alpha$ is non-zero. As $\supp{Y_\alpha} \subseteq \supp{X} \subseteq \mathcal{L}$ we can by \lemref{KY} choose $q_\alpha \in \supp{Y_\alpha}$ such that $\Sq{q_\alpha}(Y(q_\alpha))$ is a subobject of the quotient $Y_\alpha = X/X_\alpha$, and hence it has the form $\Sq{q_\alpha}(Y(q_\alpha)) = X_{\alpha+1}/X_\alpha$ for some $X_\alpha \subseteq X_{\alpha+1} \subseteq X$. Thus, \rqmlbl{3$'$} holds. Note that since $Y_\alpha(q_\alpha) \neq 0$, the inclusion $X_\alpha \subset X_{\alpha+1}$ is strict, so eventually we get $X_\lambda = X$ for some ordinal $\lambda$, and hence \rqmlbl{1} holds.\smallbreak

Below we prove that the following condition holds for every $\alpha \le \lambda$.
  \begin{eqc}
  \item[\textnormal{($*$)}] $X_\alpha(p) = X(p)$ for every $p \in \supp{X_\alpha}$.
  \end{eqc}  
Once this has been proved, the remaining property \rqmlbl{3} follows from the already established property \rqmlbl{3$'$}. Indeed, by \rqmlbl{3$'$} we have a $q_\alpha \in \supp{X/X_\alpha}$ and hence $X_\alpha(q_\alpha) \neq X(q_\alpha)$; thus condition ($*$) implies that $q_\alpha \notin \supp{X_\alpha}$. This means that $X_\alpha(q_\alpha)=0$, and consequently
\begin{equation*}
  \Sq{q_\alpha}(X(q_\alpha)/X_\alpha(q_\alpha)) \,=\,
  \Sq{q_\alpha}(X(q_\alpha)) \;.
\end{equation*}
Therefore, in view of ($*$), the property \rqmlbl{3} follows from \rqmlbl{3$'$}.\smallbreak

It remains to prove condition ($*$). We do this by transfinite induction on $\alpha$. For $\alpha=0$ there is nothing to prove as $X_0=0$ and hence $\supp{X_0} = \varnothing$. Now let $\beta \le \lambda$ be an ordinal and assume that ($*$) holds for every $\alpha<\beta$. First assume that  $\beta$ is a limit ordinal, in which case \smash{$X_\beta = \varinjlim_{\alpha<\beta} X_\alpha$} by \rqmlbl{2}. There is an equality,
\begin{equation}
  \label{eq:cup}
  \textstyle
  \supp{X_\beta} \,=\, \bigcup_{\alpha<\beta} \supp{X_\alpha}\;.
\end{equation}
Indeed, the inclusion ``$\subseteq$'' is trivial. On the other hand, if $p \in \supp{X_{\alpha_0}}$ for some $\alpha_0<\beta$, then also $p \in \supp{X_\alpha}$ for every $\alpha \ge \alpha_0$ since there is an inclusion $X_{\alpha_0} \subseteq X_\alpha$. Hence, by the induction hypothesis one gets $0 \neq X_\alpha(p) = X(p)$ for every $\alpha_0 \le \alpha < \beta$, and consequently
\begin{equation*}
  \textstyle
  X_\beta(p) 
  \,=\, 
  \varinjlim_{\,\alpha<\beta} X_\alpha(p) 
  \,=\, 
  \varinjlim_{\,\alpha_0 \le \alpha<\beta} X_\alpha(p) 
  \,=\, 
  X(p) \,\neq\, 0 \;.
\end{equation*}
This simultaneously proves the other inclusion ``$\supseteq$'' in \eqref{cup} and establishes condition ($*$) for the limit ordinal $\beta$. Finally, consider the case where $\beta = \alpha+1$ is a successor ordinal. By \rqmlbl{3$'$} there is a short exact sequence
\begin{equation}
\label{eq:ses}
  0 \longrightarrow X_\alpha \longrightarrow X_{\alpha+1} \longrightarrow \Sq{q_\alpha}(Y_\alpha(q_\alpha)) \longrightarrow 0
\end{equation}
where $Y_\alpha = X/X_\alpha$ and $q_\alpha \in \supp{Y_\alpha}$. This immediately yields:
\begin{equation*}
  \supp{X_{\alpha+1}} \,=\, \supp{X_\alpha} \,\cup\, \supp{\Sq{q_\alpha}(Y_\alpha(q_\alpha))} \,=\, \supp{X_\alpha} \,\cup\, \{q_\alpha\}\;.
\end{equation*}
For $p \in \supp{X_\alpha}$ the induction hypothesis yields $X_\alpha(p) = X(p)$, and hence $X_{\alpha+1}(p) = X(p)$ by the inclusions $X_\alpha \subseteq X_{\alpha+1} \subseteq X$. As $q_\alpha \in \supp{Y_\alpha}$ we have $X_\alpha(q_\alpha) \neq X(q_\alpha)$, so $q_\alpha$ does not belong to $\supp{X_\alpha}$ as condition ($*$) holds for the ordinal $\alpha$. This means that $X_\alpha(q_\alpha)=0$ and hence $Y_\alpha(q_\alpha) = X(q_\alpha)$. In view of this, the short exact sequence \eqref{ses} now yields
\begin{equation*}
  X_{\alpha+1}(q_\alpha) \,=\, \Sq{q}(Y_\alpha(q_\alpha))(q_\alpha) \,=\, Y_\alpha(q_\alpha) \,=\, X(q_\alpha) \;,
\end{equation*}
which establishes condition ($*$) for the successor ordinal $\beta = \alpha+1$.
\end{proof}

\begin{dfn}\label{dfn:filtration}
    Let $\mathcal{Z}$ be a class of objects in an exact category $\mathcal{E}$. We say that an object $X$ in $\mathcal{E}$ is \textbf{filtered by objects in $\mathcal{Z}$} (or it admits a \textbf{$\mathcal{Z}$-filtration}) if there exists a sequence of inflations
  \begin{equation*}
   0 = X_0 \rightarrowtail X_1 \rightarrowtail X_2 \rightarrowtail \cdots  \rightarrowtail X_\omega \rightarrowtail X_{\omega+1} \rightarrowtail \cdots \rightarrowtail X_\lambda = X
  \end{equation*}
  with the following properties:
  \begin{rqm}
  \item For every limit ordinal $\beta \le \lambda$ one has \smash{$X_\beta = \varinjlim_{\alpha<\beta} X_\alpha$}.
 \item For every successor ordinal $\alpha+1 \le \lambda$ one has $X_{\alpha+1}/X_\alpha \in \mathcal{Z}$.
  \end{rqm}
  The class of all $\mathcal{Z}$-filtered objects is denoted by $\operatorname{Filt}\text{-}\mathcal{Z}$.
\end{dfn}

\begin{lem}
  \label{lem:filt-supp}
  Let $\mathbb{L}$ be a class of objects in $Q$ and $\mathcal{Z}$ a class of objects in $\lMod{Q,A}$ such that $\supp{Z} \subseteq \mathbb{L}$ holds for every $Z \in \mathcal{Z}$. Then one has $\supp{X} \subseteq \mathbb{L}$ for every $X \in \operatorname{Filt}\text{-}\mathcal{Z}$. Consequently, $\supp{(\operatorname{Filt}\text{-}\mathcal{Z})}=\supp{\mathcal{Z}}$.
\end{lem}

\begin{proof}
  If $X$ belongs to $\operatorname{Filt}\text{-}\mathcal{Z}$, then $X$ has a filtration
  \begin{equation*}
   0 = X_0 \rightarrowtail X_1 \rightarrowtail X_2 \rightarrowtail \cdots  \rightarrowtail X_\omega \rightarrowtail X_{\omega+1} \rightarrowtail \cdots \rightarrowtail X_\lambda = X
  \end{equation*}
  satisfying the properties of Definition \ref{dfn:filtration}.\smallbreak
  
By transfinite induction we will prove that $\supp{X_\alpha} \subseteq \mathbb{L}$ for every $\alpha \le \lambda$; note that with $\alpha = \lambda$ we obtain $\supp{X} \subseteq \mathbb{L}$, as desired. For $\alpha=0$ there is nothing to prove as $X_0=0$ and hence $\supp{X_0} = \varnothing$. Now let $\beta \le \lambda$ be an ordinal and assume that $\supp{X_\alpha} \subseteq \mathbb{L}$ holds for every $\alpha < \beta$. If $\beta$ is a limit ordinal, then 
\smash{$X_\beta = \varinjlim_{\alpha<\beta} X_\alpha$} by \rqmlbl{1},  and hence
\begin{equation*}
  \textstyle
  \supp{X_\beta} \,=\, \bigcup_{\alpha<\beta} \supp{X_\alpha} \,\subseteq\, \mathbb{L} \;,
\end{equation*}
where the equality follows as in \eqref{cup} and the inclusion holds by the induction hypothesis. If $\beta = \alpha+1$ is a successor ordinal the exact sequence $0 \to X_\alpha \to X_{\alpha+1} \to X_{\alpha+1}/X_\alpha \to 0$~yields
\begin{equation}
  \label{eq:3xsupp}
  \supp{X_{\alpha+1}} 
  \,=\, 
  \supp{X_\alpha} \;\cup\; \supp{(X_{\alpha+1}/X_\alpha)}  
  \;.
\end{equation}
By the induction hypothesis we have $\supp{X_\alpha} \subseteq \mathbb{L}$. By \rqmlbl{2} one has $X_{\alpha+1}/X_\alpha \in \mathcal{Z}$ and hence $\supp{(X_{\alpha+1}/X_\alpha)} \subseteq \mathbb{L}$ by assumption. Thus \eqref{3xsupp} implies that $\supp{X_{\alpha+1}} \subseteq \mathbb{L}$.\smallbreak

We proved that $\supp{(\operatorname{Filt}\text{-}\mathcal{Z})}\subset\supp{\mathcal{Z}}$. The other inclusion $\supp{\mathcal{Z}}\subset\supp{(\operatorname{Filt}\text{-}\mathcal{Z})}$ is straightforward.
\end{proof}

Recall from Holm and J{\o}rgensen \cite[Thm.~6.5]{HJ-JLMS} the Frobenius exact subcategory ${}^{\perp_1}\lE{}$ of the Grothendieck abelian category $\lMod{Q,A}$ consisting of the bifibrant objects with respect to the model structure from \emph{loc. cit.}.
\begin{prp}
  \label{prp:filt-description}
  For the subset \mbox{$\mathcal{S}_{\mathcal{L}} = \{\Sq{q}(A) \ | \ q \in \mathcal{L} \}$} of ${}^{\perp_1}\lE{} \subseteq \lMod{Q,A}$ one has
\begin{equation*}
    \operatorname{Filt}\text{-}\mathcal{S}_{\mathcal{L}}
    \,=\,
    \{ X \in \lMod{Q,A} \;|\; 
    \supp{X} \subseteq \mathcal{L} \text{ and } 
    X(q) \text{ is a free $A$-module for every } q \in \mathcal{L}
    \} \;.
\end{equation*}
Moreover, this class is contained in ${}^{\perp_1}\lE{}$.
\end{prp}

\begin{proof}
  First assume that $X$ belongs to $\operatorname{Filt}\text{-}\mathcal{S}_{\mathcal{L}}$; that is, $X$ has a filtration
  \begin{equation}
  \label{eq:filtration-1}
   0 = X_0 \rightarrowtail X_1 \rightarrowtail X_2 \rightarrowtail \cdots  \rightarrowtail X_\omega \rightarrowtail X_{\omega+1} \rightarrowtail \cdots \rightarrowtail X_\lambda = X
  \end{equation}
 satisfying the properties of Definition \ref{dfn:filtration}.\smallbreak
  
Since one has $\supp{\Sq{q}(A)} = \{q\}$, every object in $\mathcal{S}_{\mathcal{L}}$ has its support contained in $\mathcal{L}$, so \lemref{filt-supp} implies that $\supp{X} \subseteq \mathcal{L}$.\smallbreak
  
Now fix $q \in \mathcal{L}$. By \rqmlbl{2} each quotient $X_{\alpha+1}/X_\alpha$ has the form $X_{\alpha+1}/X_\alpha = \Sq{q_\alpha}(A)$ for some $q_\alpha \in \mathcal{L}$ and hence $X_{\alpha+1}(q)/X_\alpha(q) = \Sq{q_\alpha}(A)(q) \in \{0,A\}$ by \cite[Lem.~7.10 and Prop.~7.15]{HJ-JLMS}. Thus, by evaluating the given filtration \eqref{filtration-1} at $q$, we obtain a filtration of the $A$-module $X(q)$ with quotients in $\{0,A\}$; that is, $X(q) \in \operatorname{Filt}\text{-}\{0,A\}$ in $\lMod{A}$. It now easily follows that $X(q)$ is free, cf.~the remarks following G{\"o}bel and Trlifaj \cite[Cor.~3.2.4]{GobelTrlifaj}.\smallbreak

Conversely, assume $\supp{X} \subseteq \mathcal{L}$ and that the $A$-module $X(q)$ is free for every $q \in \mathcal{L}$ (equivalently, for every $q \in Q$ since $X(q)=0$ for $q \notin \mathcal{L}$). Consider the set
\begin{equation*}
  \mathcal{X} 
  \,=\, 
  \{ \Sq{q}(X(q)) \ | \ q \in \mathcal{L} \} \;.
\end{equation*}
It follows from \prpref{filtration} that $X \in \operatorname{Filt}\text{-}\mathcal{X}$, so to show that $X$ is in $\operatorname{Filt}\text{-}\mathcal{S}_{\mathcal{L}}$, it suffices to prove the inclusion
\begin{equation}
  \label{eq:inclusion}
  \operatorname{Filt}\text{-}\mathcal{X}
  \,\subseteq\, 
  \operatorname{Filt}\text{-}\mathcal{S}_{\mathcal{L}} \;.
\end{equation}
Recall from \cite[Prop.~3.12]{HJ-JLMS} that $\lMod{Q,A}$ is a Grothendieck category, so the theory developed in {\v{S}}t'ov{\'\i}{\v{c}}ek \cite{Sto2013a} applies to this category. Since $\mathcal{S}_{\mathcal{L}}$ is a set, $\operatorname{Filt}\text{-}\mathcal{S}_{\mathcal{L}}$ is by definition a deconstructible class, see \cite[Def.~1.4]{Sto2013a}. Hence, to prove the inclusion \eqref{inclusion}, it suffices by \cite[Lem.~1.6]{Sto2013a} to argue that $\mathcal{X} \subseteq \operatorname{Filt}\text{-}\mathcal{S}_{\mathcal{L}}$ holds. By assumption, each module $X(q)$ is free, so to finish the proof we must show that the object $\Sq{q}(A^{(I)})$ belongs to $\operatorname{Filt}\text{-}\mathcal{S}_{\mathcal{L}}$ for every $q \in \mathcal{L}$ and every (index) set $I$. However, this is trivial as 
\begin{equation*}
  \Sq{q}(A^{(I)}) 
  \,=\, 
  \Sq{q}(A)^{(I)} 
  \,\in\, 
  \operatorname{Filt}\text{-}\mathcal{S}_{\mathcal{L}} \;,
\end{equation*}
where ``$\in$'' follows as $\operatorname{Filt}\text{-}\mathcal{S}_{\mathcal{L}}$ is closed under coproducts, again by \cite[Lem.~1.6]{Sto2013a}.

The fact that $\operatorname{Filt}\text{-}\mathcal{S}_{\mathcal{L}}$ is contained in ${}^{\perp_1}\lE{}$ follows from \cite[Prop.~3.2(a)]{HJ-TAMS}.
\end{proof}

Note that the ``in particular'' statement in the next result gives a partial converse to the second inclusion in \cite[Prop.~3.2(a)]{HJ-TAMS}. A special case of the statement is that every bounded below complex of projective modules is automatically semi-projective; this is of course well-known, see for instance \cite[Example~5.5]{MR3279365}.

\begin{cor}
  \label{cor:filt-description-summand}
  As in \prpref{filt-description}, consider the set \mbox{$\mathcal{S}_{\mathcal{L}} = \{\Sq{q}(A) \ | \ q \in \mathcal{L} \}$}. For every $X \in \lMod{Q,A}$ the following conditions are equivalent:
\begin{eqc}
\item $X$ is a direct summand of an object in $\operatorname{Filt}\text{-}\mathcal{S}_{\mathcal{L}}$.
\item $\supp{X} \subseteq \mathcal{L}$ and $X(q) \in \lPrj{A}$ for every $q \in \mathcal{L}$.
\end{eqc}  
In particular, every object $X$ that satisfies condition \eqclbl{ii} belongs to ${}^{\perp_1}\lE{}$.
\end{cor}

\begin{proof}
  If $X$ is a direct summand of an object $F \in \operatorname{Filt}\text{-}\mathcal{S}_{\mathcal{L}}$ then, in particular, there is an inclusion $\supp{X} \subseteq \supp{F}$ and $X(q)$ is a direct summand of $F(q)$ for every $q \in Q$. Hence the implication \eqclbl{i}$\;\Rightarrow\;$\eqclbl{ii} follows immediately from \prpref{filt-description}.\smallbreak
  
Conversely, assume that $X$ satisfies condition \eqclbl{ii}. As in the proof of \prpref{filt-description} we have $X \in \operatorname{Filt}\text{-}\mathcal{X}$ where $\mathcal{X} = \{ \Sq{q}(X(q)) \ | \ q \in \mathcal{L} \}$. As also noted in that proof, the class
$\operatorname{Filt}\text{-}\mathcal{S}_{\mathcal{L}}$ is deconstructible. Hence {\v{S}}t'ov{\'\i}{\v{c}}ek \cite[Prop.~2.9(1)]{Sto2013a} shows that the class, say, $\mathcal{F}$, of objects that satisfy condition \eqclbl{i} is deconstructible too. To see that $X$ satisfies condition \eqclbl{i} it suffices to show the inclusion $\operatorname{Filt}\text{-}\mathcal{X} \subseteq \mathcal{F}$, and to that end it is by \cite[Lem.~1.6]{Sto2013a} is enough to argue that $\mathcal{X} \subseteq \mathcal{F}$. Thus, we must prove that every object of the form $\Sq{q}(X(q))$ where $q \in \mathcal{L}$ satisfies condition \eqclbl{i}. By assumption, the $A$-module $X(q)$ is projective, and hence it is a direct summans of a free $A$-module $A^{(I)}$ where $I$ is some set. Thus, $\Sq{q}(X(q))$ is a direct summand of $\Sq{q}(A^{(I)})$, and the latter object belongs to $\operatorname{Filt}\text{-}\mathcal{S}_{\mathcal{L}}$, as already demonstrated in the proof of \prpref{filt-description}. Therefore, $\Sq{q}(X(q))$ satisfies condition \eqclbl{i}, as desired.\smallbreak

The ``in particular'' statement follows immediately from the last assertion in \prpref{filt-description} and the fact that class ${}^{\perp_1}\lE{}$ is closed under direct summands in $\lMod{Q,A}$.
\end{proof}

 Saor{\'{\i}}n and {\v{S}}t'ov{\'\i}{\v{c}}ek \cite[Def.~2.6]{SaorinStovicek} introduced the notion of \textbf{efficient} exact categories. In such categories the small object argument guarantees the completeness of cotorsion pairs that are generated by a set of objects. {\v{S}}t'ov{\'\i}{\v{c}}ek in \cite[Def.~3.11]{Sto2013b} also introduces exact categories of Grothendieck type, which is an even stronger notion than  efficiency. Indeed, $\mathscr{C}$ is of \textbf{Grothendieck type} if and only if it is efficient (in the sense of \cite[Def.~3.4]{Sto2013b} which is a stronger version of \cite[Def.~2.6]{SaorinStovicek}) and it is \textbf{deconstructible}, meaning that there is a set (as opposed to a class) $\mathcal{S}$ of objects in $\mathscr{C}$ such that $\mathscr{C} = \operatorname{Filt}\text{-}\mathcal{S}$.

\begin{lem}
  \label{lem:perp-E-is-efficient}
  The Frobenius exact category ${}^{\perp_1}\lE{}$ is of Grothendieck type.
\end{lem}

\begin{proof}
  Clearly, the class ${}^{\perp_1}\lE{}$ is closed under direct summands (= retracts) in $\lMod{Q,A}$. Thus, to prove the assertion, it suffices by \cite[Thm.~3.16]{Sto2013b} to prove that ${}^{\perp_1}\lE{}$ is deconstructible.  Recall from \cite[Thm.~5.5]{HJ-JLMS} that $({}^{\perp_1}\lE{},\lE{})$ is a cotorsion pair generated by some set $\mathcal{S}$; that is, $\mathcal{S}^{\perp_1} = \lE{}$. Note that the Grothendieck category $\lMod{Q,A}$ has a projective generator, e.g. $G = \bigoplus_{q \in Q}\Fq{q}(A)$ by \cite[Prop.~3.12(a)]{HJ-JLMS}. As the set \mbox{$\mathcal{S} \cup \{G\}$} also generates the cotorsion pair $({}^{\perp_1}\lE{},\lE{})$, it follows from \cite[Prop.~1.7 and Rem.~1.8]{Sto2013a} that ${}^{\perp_1}\lE{}$ is deconstructible.
\end{proof}

\begin{thm}
  \label{thm:cotorsion}
  In the Frobenius exact category ${}^{\perp_1}\lE{}$, the cotorsion pair $(\mathcal{X}_\mathcal{L},\mathcal{Y}_\mathcal{L})$ generated by the set $\mathcal{S}_{\mathcal{L}} = \{\Sq{q}(A) \ | \ q \in \mathcal{L} \}$ can be described as follows:
\begin{prt}

\item An object $X \in {}^{\perp_1}\lE{}$ belongs to $\mathcal{X}_\mathcal{L}$ if and only if $X$  lies in $\operatorname{Add}(\lPrj{Q,A},\operatorname{Filt}\text{-}\mathcal{S}_\mathcal{L})$.

\item An object $Y \in {}^{\perp_1}\lE{}$ belongs to $\mathcal{Y}_\mathcal{L}$ if and only if $\cH[1]{q}(Y)=0$ for every $q \in \mathcal{L}$.
\end{prt}  
Moreover, the cotorsion pair $(\mathcal{X}_\mathcal{L},\mathcal{Y}_\mathcal{L})$ in ${}^{\perp_1}\lE{}$ is complete.
\end{thm}

\begin{proof}
 We know from \lemref{perp-E-is-efficient} that the exact category ${}^{\perp_1}\lE{}$ is of Gro\-then\-dieck type and hence also efficient. Since ${}^{\perp_1}\lE{}$ is Frobenius, it has enough projectives; in fact, the class of projectives in the exact category ${}^{\perp_1}\lE{}$ is precisely class $\lPrj{Q,A}$ of projectives in the Grothendieck abelian category $\lMod{Q,A}$, see \cite[Thm.~6.5]{HJ-JLMS}. Hence Saor{\'{\i}}n and {\v{S}}t'ov{\'\i}{\v{c}}ek \cite[Cor.~2.15(3)]{SaorinStovicek} shows that the cotorsion pair $(\mathcal{X}_\mathcal{L},\mathcal{Y}_\mathcal{L})$ in ${}^{\perp_1}\lE{}$
 generated by the set $\mathcal{S}_{\mathcal{L}}$ is complete and that an object $X \in {}^{\perp_1}\lE{}$ belongs to $\mathcal{X}_\mathcal{L}$ if and only if it is a direct summand of an object $C \in {}^{\perp_1}\lE{}$ appearing in a short exact sequence \begin{equation}\label{eq:def_of_XL}
 0 \to P \to C \to F \to 0\end{equation}
 with $P \in \lPrj{Q,A}$ and $F \in \operatorname{Filt}\text{-}\mathcal{S}_{\mathcal{L}}$. Notice however, that the short exact sequence \eqref{def_of_XL} is a conflation in the Frobenius exact category ${}^{\perp_1}\mathscr{E}$, thus it splits. We conclude that $X$ lies in $\operatorname{Add}(\lPrj{Q,A},\operatorname{Filt}\text{-}\mathcal{S}_\mathcal{L})$. The converse follows by Corollary \ref{cor:filt-description-summand}.\smallbreak
 
 To prove part \prtlbl{b}, note that the first two equalities below hold by definition while the third one follows from \cite[proof of Thm.~D]{HJ-TAMS}:
\begin{align*}
  \mathcal{Y}_\mathcal{L}
  &\,=\,
  (\mathcal{S}_\mathcal{L})^{\perp_1}
  \\
  &\,=\,
  \{ Y \in {}^{\perp_1}\lE{} \ | \ \Ext{Q,A}{1}{\Sq{q}(A)}{Y}=0 \text{ for every } q \in \mathcal{L} \}
  \\
  &\,=\,
  \{ Y \in {}^{\perp_1}\lE{} \ | \ \cH[1]{q}(Y)=0 \text{ for every } q \in \mathcal{L} \} \;. \qedhere
\end{align*}
\end{proof}

At this point it is important to stress that we actually denote by ${}^{\perp_1}\underline{\lE{}}$ the subcategory $\underline{{}^{\perp_1}\lE{}}$, since the Ext-orthogonality is taken in the model and not in the homotopy category. This abbreviation is solely for aesthetic reasons.

\begin{prp}
  \label{prp:image}
  Consider the cotorsion pair $(\mathcal{X}_\mathcal{L},\mathcal{Y}_\mathcal{L})$ in the Frobenius exact category ${}^{\perp_1}\lE{}$ from \thmref{cotorsion}.
\begin{prt}

\item The image $(\underline{\mathcal{X}}_\mathcal{L},\underline{\mathcal{Y}}_\mathcal{L})$ of the cotorsion pair in the stable category ${}^{\perp_1}\underline{\lE{}}$ is 
\begin{align*}
  \underline{\mathcal{X}}_\mathcal{L}
  &\,=\,
     \left\{ X \in {}^{\perp_1}\underline{\lE{}} \;\left|\mspace{-5mu} 
    \begin{array}{l}
    \text{There exists an object $P$ with $\supp{P} \subseteq \mathcal{L}$ and $P(q) \in \lPrj{A}$} \\
    \text{for every $q \in \mathcal{L}$ and an isomorphism $P \cong X$ in ${}^{\perp_1}\underline{\lE{}}$}
    \end{array}
    \mspace{-8mu}
    \right.
    \right\} ,
  \\  
  \underline{\mathcal{Y}}_\mathcal{L}
  &\,=\,
  \{ Y \in {}^{\perp_1}\underline{\lE{}} \ | \ \cH[1]{q}(Y)=0 \text{ for every } q \in \mathcal{L} \} \;.  
\end{align*}

\item The image $(\mathfrak{X}_\mathcal{L},\mathfrak{Y}_\mathcal{L})$ of the cotorsion pair in the $Q$-shaped derived category $\QSD{Q}{A}$~is
\begin{align*}
  \phantom{XXX}
  \mathfrak{X}_\mathcal{L}
  &\,=\,
    \left\{ X \in \QSD{Q}{A} \;\left|\mspace{-5mu} 
    \begin{array}{l}
    \text{There exists an object $P$ with $\supp{P} \subseteq \mathcal{L}$ and $P(q) \in \lPrj{A}$} \\
    \text{for every $q \in \mathcal{L}$ and an isomorphism $P \cong X$ in $\QSD{Q}{A}$}
    \end{array}
    \mspace{-8mu}
    \right.
    \right\} ,
  \\  
  \mathfrak{Y}_\mathcal{L}
  &\,=\,
  \{ Y \in \QSD{Q}{A} \ | \ \cH[1]{q}(Y)=0 \text{ for every } q \in \mathcal{L} \} \;.  
\end{align*}
\end{prt}  
\end{prp}

\begin{proof}
 Since objects in $\lPrj{Q,A}$ are zero in ${}^{\perp_1}\underline{\lE{}}$, part \prtlbl{a} follows from \thmref{cotorsion},   \prpref{filt-description}, and \corref{filt-description-summand}. Part \rqmlbl{b} follows from \rqmlbl{a}.
\end{proof}

\begin{thm}\label{thm: t-structure}
    The cotorsion pair $(\mathcal{X}_\mathcal{L},\mathcal{Y}_\mathcal{L})$ in the Frobenius exact category ${}^{\perp_1}\lE{}$ corresponds to a t-structure $(\Omega\,\underline{\mathcal{X}}_\mathcal{L},\underline{\mathcal{Y}}_\mathcal{L})$ in the stable category ${}^{\perp_1}\underline{\lE{}}$.
\end{thm}

\begin{proof}
    By Saor\'{i}n and {\v{S}}t'ov{\'\i}{\v{c}}ek theorem \cite[Prop.~3.9]{SaorinStovicek} one needs to prove that $(\mathcal{X}_\mathcal{L},\mathcal{Y}_\mathcal{L})$ is a complete and cohereditary cotorsion pair. The completeness of the cotorsion pair is proven in Theorem \ref{thm:cotorsion}.\smallbreak

    To prove that the aforementioned cotorsion pair is also cohereditary, it suffices to show that the $\mathcal{X_{\mathcal{L}}}$ is closed under co-syzygies, that is, $\Omega^{-1}\mathcal{X_{\mathcal{L}}}\subseteq\mathcal{X_{\mathcal{L}}}$. Consider an object $X$ in $\mathcal{X_{\mathcal{L}}}$, then by \cite[Prop.~5.7]{HJ-TAMS} there is a short exact sequence in $\lMod{Q,A}$ (and in particular in ${}^{\perp_1}\lE{}$)
    \[
    0\longrightarrow X\longrightarrow \mathbb{G}(X)\longrightarrow\mathbb{C}(X)\longrightarrow0,
    \]
    which is object-wise split, $\mathbb{G}(X)=\prod_{q\in\supp{X}}G_q(X(q))$ and $\mathbb{C}(X)$ is the cokernel of the morphism $X\rightarrow \mathbb{G}(X)$, that is $\mathbb{C}(X)=\Omega^{-1}(X)$. Moreover, by \cite[Proof~of~Prop.~5.8]{HJ-TAMS}, one has that 
    \[
    \supp{\mathbb{G}(X)}\subseteq\bigcup_{q\in\supp{X}}\mathsf{N}_{-}(q)
    \]
    and by Definition \ref{dfn:L} since every $q\in\supp{X}$ is in $\mathcal{L}$, it follows that $\supp{\mathbb{G}(X)}\subseteq\mathcal{L}$. Assume now that there is a $p\in Q$ such that $\mathbb{C}(X)(p)\neq0$. By the exactness of the split short exact sequence 
    \[
    0\longrightarrow X(p)\longrightarrow \mathbb{G}(X)(p)\longrightarrow\mathbb{C}(X)(p)\longrightarrow0,
    \]
    in $\lMod{A}$, it follows that $\mathbb{G}(X)(p)\neq0$ and hence $\supp{\mathbb{C}(X)}\subseteq\supp{\mathbb{G}(X)}\subseteq\mathcal{L}$. This proves that the cotorsion pair in question is cohereditary and by \cite[Prop.~3.9]{SaorinStovicek}, the pair $(\Omega\,\underline{\mathcal{X}}_\mathcal{L},\underline{\mathcal{Y}}_\mathcal{L})$ forms a t-structure in the stable category ${}^{\perp_1}\underline{\lE{}}$.
\end{proof}

\section{Silting \texorpdfstring{$t$}{t}-structures}\label{sec:tilting}
The t-structures created in the previous Section, under mild assumptions enjoy certain pleasant properties. In particular, one can show that they are in fact induced by certain silting objects. We provide a brief introduction to certain relevant notions for what follows.
\smallbreak

 An object $M$ in $\mathcal{T}$ is called \textbf{silting} (according to Aihara and Iyama \cite[Def.~4.1]{aihara_iyama_silting}) if it is compact, it generates $\mathcal{T}$, in the sense that $M^{\perp_\mathbb{Z}}=0$, and $\Hom{\mathcal{T}}{M}{\Sigma^iM}=0$ for all $i>0$. In a triangulated category $\mathcal{T}$ with arbitrary coproducts, a silting object induces a t-structure $(M^{\perp_{>0}},M^{\perp_{\le0}})$ which we call \textbf{silting t-structure} (see \emph{loc. cit.} Thm. 4.3 and Cor. 4.7). A silting object $M$ is called \textbf{tilting} if $\operatorname{Add}(M)\subseteq\mathcal{H}_M$, where the latter denotes the heart of the silting t-structure. Later, \cite{nicolas2019silting} and \cite{psaroudakis2018realisation} defined silting objects in triangulated categories based on silting t-structures as their defining property and without requiring the object to be compact. Often in the literature these objects are called \textbf{large silting objects}. We will work with the definition of Aihara and Iyama.

\begin{dfn}\label{dfn: boundary}
   Let $\mathcal{L},\mathcal{R}$ be an admissible partition of $Q$. We define the \textbf{boundary} of $\mathcal{L}$ (respectively of $\mathcal{R}$) in the following sense
    \[
    \begin{aligned}
    \partial\mathcal{L}&:=\{q\in\mathcal{L}\,|\, \text{There exists some object } p\in\mathcal{R} \text{ such that } Q(q,p)\neq0\}\\
    \partial \mathcal{R}&:=   \{q\in\mathcal{R}\,|\, \text{There exists some object } p\in\mathcal{L} \text{ such that } Q(p,q)\neq0\}
    \end{aligned}
    \]
\end{dfn}

Notice that by the assumptions of the Setup \cite[Setup~2.1]{HJ-Abel}, mainly the local boundedness and the nilpotency of the pseudo-radical, imply that the boundaries of $\mathcal{L}$ and $\mathcal{R}$ are both finite. \smallbreak

Fix a ground field $\Bbbk$. We consider that full subcategory $\partial\mathcal{L}$ of $Q$ whose objects are the objects of the boundary of $\mathcal{L}$ and the morphisms are the same as in $Q$, restricted to the objects of $\partial\mathcal{L}$. From this point onwards, we assume also that $Q$ is the path category of a quiver, possibly modulo certain relations, and thus $\partial\mathcal{L}$ is the path category of a finite sub-quiver of $Q$. Consider 
\[
\lMod{\partial\mathcal{L}}:=\operatorname{Fun}_\Bbbk(\partial\mathcal{L},\, 
\lMod{\Bbbk}),
\]
to be the category of representations of $\partial\mathcal{L}$ with values in $\lMod{\Bbbk}$. 

\begin{lem}
    The category $\lMod{\partial\mathcal{L}}$ is equivalent to the module category $\lMod{\Lambda}$ for a finite-dimensional $\Bbbk$-algebra $\Lambda$. In particular, $\Lambda$ is of finite global dimension.
\end{lem}

\begin{proof}
By the conditions imposed to $\mathcal{L}$ in Definition \ref{dfn:L}, one has that $\partial\mathcal{L}$ is a finite (since $Q$ is locally bounded) acyclic (Definition \ref{dfn:L}(2)) quiver and thus its category of representations is equivalent to the module category of the finite-dimensional $\Bbbk$-algebra $\Lambda:=\Bbbk(\partial\mathcal{L})/I$, where $I$ is the \emph{admissible} ideal of relations obtained by restricting the relations of $Q$ to $\partial\mathcal{L}$. For the admissibility of $I$ it suffices to notice that since $\partial\mathcal{L}$ is finite and acyclic there exists a maximal length of a path in $\partial\mathcal{{L}}$. Consequently, denoting by $J$ the arrow ideal of the path algebra $\Bbbk(\partial\mathcal{L})$, we conclude that $J$ is nilpotent for some integer $m\ge2$. In particular, one has  $J^m\subseteq I$. On the other hand, $I$ is obtained by restricting the relations of $Q$ to the subquiver $\partial\mathcal{L}$. Since the relations of $Q$ lie in the square of the radical ideal every relation has length at least two, meaning $I\subseteq J^2$.\smallbreak

Finally, since $\partial\mathcal{L}$ is a finite acyclic quiver with $\ell$ vertices for some positive integer $\ell$, by \cite[Prop.~3.1.7]{Derksen_Weyman_quivers} one has that the projective dimension of any $\Lambda$-module $M$ (and thus also for the simple modules) is finite. More concretely, one has for every $\Lambda$-module $M$, $\operatorname{pd}_\Lambda(M)\le\ell-1$. It is now clear that $\Lambda$ has finite global dimension.
\end{proof}

Notice that, the category $\lMod{\partial\mathcal{L}}$ is Krull-Schmidt and one can compute precisely the indecomposable projective representations $P_i$ of $\partial\mathcal{L}$, which are also going to be finitely many.
\
Note that every functor in $\operatorname{Fun}_\Bbbk(\partial\mathcal{L},\, 
\lMod{\Bbbk})$ takes values in projective/injective $\Bbbk$-modules. Denote by 
\[
\vartheta\colon\lMod{\partial\mathcal{L}}\longrightarrow\lMod{Q}
\]
the fully faithful functor that maps a representation $M$ of $\partial\mathcal{L}$ to the representation $X=\vartheta(M)$ of $Q$ with $X(q)=M(q)$ in the vertices $q\in\partial\mathcal{L}$ and $X(q)=0$ elsewhere. We call $\vartheta$ the \textbf{inclusion functor} of $\lMod{\partial\mathcal{L}}$ to $\lMod{Q}$. We are also interested in the functor 
\[
\psi\colon \lMod{\partial\mathcal{L}}\longrightarrow\QSD{Q}{\Bbbk}
\]
defined as the composition of the inclusion functor $\vartheta$ with the canonical functor $\lMod{Q}\longrightarrow\QSD{Q}{\Bbbk}$. The following Lemma is inspired by the ideas of Marks and Vit\'oria in \cite[Lem.~2.4]{marks_vitoria_silting}, however the main idea of the proof is similar. We include the proof for completeness. 

\begin{lem}\label{lem: homological embedding}
    The inclusion functor $\vartheta\colon\lMod{\partial\mathcal{L}}\longrightarrow\lMod{Q}$ is a homological embedding, in the sense that, for every $M, N$ in $\lMod{\partial\mathcal{L}}$ and every $i\ge0$, there is a natural isomorphism
    \[
    \Ext{\partial\mathcal{L}}{i}{M}{N}\,\cong\,\Ext{Q}{i}{\vartheta(M)}{\vartheta(N)}.
    \]
    Moreover, there is a natural isomorphism
    \[
    \Ext{\partial\mathcal{L}}{i}{M}{N}\,\cong\,\Hom{\QSD{Q}{\Bbbk}}{\psi(M)}{\Sigma^i(\psi(N))}
    \]
\end{lem}

\begin{proof}
    It is straightforward to check that the functor $\vartheta$ is fully faithful and exact and thus it induces an injective map $\Ext{\partial\mathcal{L}}{i}{M}{N}\,\longrightarrow\,\Ext{Q}{i}{\vartheta(M)}{\vartheta(N)}$ for any $i>0$ and any $M,N\in\lMod{\partial\mathcal{L}}$. In order to show that this map is also surjective, consider the $i$-Yoneda extension
    \[
    \xi\,\colon=\quad 0\longrightarrow\vartheta(M)\longrightarrow E_1\longrightarrow\cdots\longrightarrow E_i\longrightarrow\vartheta(N)\longrightarrow 0.
    \]

    It suffices to show that up to equivalence in $\Ext{Q}{i}{\vartheta(M)}{\vartheta(N)}$ (see for example \cite{Wei}), one can choose for each $j=1,\ldots,i$, objects $E_j$ in $\lMod{Q}$ that are only supported in $\partial\mathcal{L}$. For this purpose, define $\xi^*$ to be the exact sequence in $\lMod{Q}$ obtain by $\xi$ be setting $E_j(q)=0$ for all $q\in \mathcal{R}$ and $\xi^{**}$ to be the exact sequence obtained by $\xi^*$ by setting $E_j(q)=0$ for all $q\in \mathcal{L}\smallsetminus\partial\mathcal{L}$. One can show that there are maps $\xi\longrightarrow\xi^*\longleftarrow\xi^{**}$ of exact sequences, implying that $\xi$ and $\xi^{**}$ represent the same element in $\Ext{Q}{i}{\vartheta(M)}{\vartheta(N)}$.\smallbreak

    Moreover, every object in $\lMod{\partial\mathcal{L}}$ bi-fibrant with respect to the projective (and injective) model of the $Q$-shaped derived category (see \cite[Thm.~E]{HJ-TAMS}). Therefore, an application of \cite[Prop.~4.1]{HJ-TAMS} together with the natural isomorphism obtained by the first part of the Lemma, gives the desired result.
\end{proof}

From now on we call $\Lambda=\Bbbk(\partial\mathcal{L})/I$ the \textbf{boundary algebra} with respect to the admissible partition $(\mathcal{L},\mathcal{R})$ and we identify $\lMod{\partial\mathcal{L}}$ with $\lMod{\Lambda}$, for simplicity in our notation. Relying on the fact that one can do homological algebra in $\lMod{\Lambda}$ in a compatible way with the homological algebra of $\lMod{Q}$, we prove the following Theorem about the existence of silting objects in $\QSD{Q}{\Bbbk}$.

{\color{black}
\begin{thm}\label{thm:tilting} 
    The functor $\psi\colon \lMod{\Lambda}\longrightarrow\QSD{Q}{\Bbbk}$ maps tilting $\Lambda$-modules to silting objects in $\QSD{Q}{\Bbbk}$.
\end{thm}

\begin{proof}
    Let $T$ be a tilting $\Lambda$-module. In Lemma \ref{lem: homological embedding},
    one has that there are no positive self extensions of $\psi(T)$ in $\QSD{Q}{\Bbbk}$, since $T$ is a tilting $\Lambda$-module. It remains to show that $\psi(T)$ is a compact generator of $\QSD{Q}{\Bbbk}$. Notice that, since $T$ considered as a functor $\partial\mathcal{L}\longrightarrow \lMod{\Bbbk}$ takes finitely many non-zero values in the category of finitely presented $\Bbbk$-vector spaces ${}_\Bbbk\operatorname{prj}$, its image $\psi(T)$ is a compact (in fact perfect in the sense of \cite[Thm.~C]{HJ-TAMS}) object in $\QSD{Q}{\Bbbk}$.\smallbreak 
    
    For the generation, we prove that the smallest localising subcategory that contains the image of $T$ is the whole $\QSD{Q}{\Bbbk}$, that is $\operatorname{Loc}(\psi(T))\,=\,\QSD{Q}{\Bbbk}$. 
    For this purpose we show that the set $\{S_q\in\QSD{Q}{\Bbbk}\,|\,q\in Q\}$ is contained in $\operatorname{Loc}(\psi(T))$, by considering the cases of stalks supported in $\partial\mathcal{L}$, $\mathcal{L}$ and $\mathcal{R}$, respectively.\smallbreak
    
    \noindent\textbf{\textsf{Boundary stalks:}} By the definition of tilting modules, there exists a finite co-resolution 
    \begin{equation}\label{eq:tilting coresol}
    0\longrightarrow\Lambda\longrightarrow T_0\longrightarrow T_1\longrightarrow\cdots\longrightarrow T_n\longrightarrow 0
    \end{equation}
    of the regular module $\Lambda$, where $T_i$ lie in the additive closure $\operatorname{add}(T)$ of $T$, that is they are direct summands of finite direct sums of $T$. Applying the functor $\psi$ to the co-resolution \eqref{tilting coresol}, one can realise the image of the regular module $\psi(\Lambda)$ as a (co)cone of a map between direct summands of $\psi(T)$, thus $\psi(\Lambda)$ lies in $\operatorname{Loc}(\psi(T))$. In particular the image of any projective $\Lambda$-module $P$ lies in $\operatorname{Loc}(\psi(T))$. \smallbreak

    We claim that the image $\psi(M)$ of any finitely presented $\Lambda$-module $M$ in ${}_\Lambda\operatorname{mod}$, lies in $\operatorname{Loc}(\psi(T))$. For this purpose, let $M$ be a finitely presented $\Lambda$-module and let 
    \begin{equation}\label{eq:projresol}
    0\longrightarrow P_m\longrightarrow P_{m-1}\longrightarrow\cdots\longrightarrow P_1\longrightarrow P_0\longrightarrow M\longrightarrow 0\,,
    \end{equation}
    be a projective resolution of $M$. Recall that the global dimension of $\Lambda$ is finite and thus the projective resolution \eqref{projresol} is bounded. Moreover, the resolution \eqref{projresol} of $M$ gives rise to a family of short exact sequences
    \[
    0\longrightarrow\Omega^j(M)\longrightarrow P_{j-1}\longrightarrow\Omega^j(M)\longrightarrow0\,,\quad j=1,\ldots,m\,,
    \]
    where $\Omega^j(M)$, $j=1,\ldots,m$, denotes the $j$-th syzygy of $M$. Notice that, $\Omega^m(M)$ is a projective $\Lambda$-module and thus it lies $\operatorname{Loc}(\psi(T))$ along with the rest of the projective $\Lambda$-modules that appear in the resolution \eqref{projresol}. Inductively, one can show that each syzygy of $M$ and hence also $M$ itself lie in $\operatorname{Loc}(\psi(T))$.
    In particular, the image of every simple $\Lambda$-module lies in $\operatorname{Loc}(\psi(T))$, and thus (up to weak equivalence) stalk functors $S_q$ with $q\in\partial\mathcal{L}$ lie in the localising subcategory as well.\smallbreak

    \noindent\textbf{\textsf{Stalks in} $\mathcal{L}$:} Define a \emph{length function} $h_0\colon \mathcal{L}\to\mathbb{Z}_{\ge0}$ by
    \[
    h_0(q)=\operatorname{max}\{n\ge0\,|\, q=q_0\to q_1\to q_2\to\ldots\to q_n,\text{ with }q_n\in\partial\mathcal{L}\}
    \]
    where $q=q_0\to q_1\to q_2\to\ldots\to q_n$ is a chain of non-zero pseudo-radical morphisms in $\mathcal{L}$ and it is finite by \ref{dfn:L}(2).\smallbreak    
    
    Notice that, $h_0$ takes values in the finite set $\{0,1,\ldots,N-1\}$, where $N$ denotes the nilpotency power of the pseudo-radical of $Q$. Moreover, for every $q$ in the boundary $\partial\mathcal{L}$, the length $h_0(q)=0$. We show that every stalk $S_q$ for $q\in\mathcal{L}$ such that $h_0(q)\le N-1$ lies in $\operatorname{Loc}(\psi(T))$ by induction to $h_0(q)$. The base step $h_0(q)=0$ is already shown above. Define for an integer $0< n \le N-1$ the set
    \[
    \mathcal{L}_0^{<n}=\{q\in\mathcal{L}\,|\, h_0(q)<n\}\,,
    \]
     of all vertices of length at most $n$, and assume that $S_q$ lies in $\operatorname{Loc}(\psi(T))$ for all $q\in\mathcal{L}_0^{<n}$. Consider now all vertices $q_i\in\mathcal{L}$, $i\in I$ for a finite set $I$ (by local boundedness), such that $h_0(q_i)=n$. By \cite[Con.~5.6 and Prop.~5.7]{HJ-TAMS}, for the object $S=\bigoplus_{i\in I}S_{q_i}$ in $\lMod{Q}$ there is a short exact sequence 
    \begin{equation}\label{eq:f-ses}
    \mathfrak{f}\colon\quad0\longrightarrow \mathbb{K}(S)\longrightarrow\mathbb{F}(S)\longrightarrow S\longrightarrow0\,,
    \end{equation}
    where $\mathbb{K}(S)=\bigoplus_{i\in I}\mathbb{K}(S_{q_i})$ and $\mathbb{F}(S)=\bigoplus_{i\in I}\mathbb{F}(S_{q_i})$. By the dual argument of the proof Theorem \ref{thm: t-structure}, one has that $\supp{\mathbb{K}(S)}$ is contained in $\bigcup_{i\in I}\mathsf{N}_{+}(q_i)$. In particular, since $h_0(q_i)=n$, one has that $\supp{\mathbb{K}(S)}$ is in fact contained in $\mathcal{L}_0^{<n}$, since $\supp{\mathbb{K}(S)}\subseteq\mathsf{N}_{+}(q)\setminus\bigcup_{i\in I}\{q_i\}$. This follows from the description of the functors $S_{q_i}$ and $\mathbb{F}(S_{q_i})$ in \cite{HJ-JLMS} and more concretely, the short exact sequence \eqref{f-ses} for each summand is of the form
    \[
    0\to \mathfrak{r}(q_i,-)\to Q(q_i,-)\to Q(q_i,-)/\mathfrak{r}(q_i,-)\to0,\quad i\in I
    \]
    Notice however that since $Q$ is acyclic by the Blanket Assumption in Section \ref{sec:t-structures}, $\mathfrak{r}(q_i,q_i)=0$. Thus, $\mathbb{K}(S_{q_i})(q_i)=0$ and $\supp{\mathbb{K}(S_{q_i})}\subseteq\mathsf{N}_{+}(q_i)\setminus\{q_i\}$ for every $i\in I$. Denote for brevity by $K_i$ the kernel $\mathbb{K}(S_{q_i})$. Notice that the support of $K_i$ is finite and $K_i$ takes values in finitely presented $\Bbbk$-vector spaces. Thus it follows by \prpref{filtration} that there exists a \emph{finite} filtration 
     \begin{equation*}
   0 = K_i^0 \rightarrowtail K_i^1 \rightarrowtail K_i^2 \rightarrowtail \cdots  \rightarrowtail K_i^r = K_i
  \end{equation*}
  for some non-negative integer $r\ge0$, such that for every $j\le r$ there is a short exact sequence
  \[
  0\longrightarrow K_i^j\longrightarrow K_i^{j+1}\longrightarrow S_{q_j}\longrightarrow0\,
  \]
  for some object $q_{j}\in\supp{K}\setminus\supp{K_i^j}.$ In particular, starting from the short exact sequence
  \[
  0\longrightarrow K_i^0\longrightarrow K_i^{1}\longrightarrow S_{q_0}\longrightarrow0\,
  \]
  and proceeding inductively, one can show that the images of each $K_i^j$ lie in $\operatorname{Loc}(\psi(T))$ and so does $K_i$ since it is a finite extension of $K_i^j\,$'s, and thus also $K=\bigoplus_{i\in I}K_i$ lies in $\operatorname{Loc}(\psi(T))$. Applying $\psi$ to the sequence $\mathfrak{f}$ yields a triangle in $\QSD{Q}{\Bbbk}$, such that $\mathbb{K}(S)$ and $\mathbb{F}(S)$ (which is isomorphic to zero in the triangulated category) are in $\operatorname{Loc}(\psi(T))$, and thus so is $S$. Since $\operatorname{Loc}(\psi(T))$ is closed under summands, this proves that every stalk $S_q$ for $q\in\mathcal{L}$ such that $h_0(q)\le N-1$ lies in $\operatorname{Loc}(\psi(T))$. We denote by $\mathcal{L}_0^{\le N-1}$ the set of all such vertices.\smallbreak

  Define now another length function $h_1\colon \mathcal{L}\to\mathbb{Z}_{\ge0}$ by \[
    h_1(q)=\operatorname{max}\{n\ge0\,|\, q=q_0\to q_1\to q_2\to\ldots\to q_n,\text{ with }q_n\in\mathcal{L}_0^{\le N-1}\}
    \]
    where $q=q_0\to q_1\to q_2\to\ldots\to q_n$ is a finite chain of non-zero pseudo-radical morphisms in $\mathcal{L}$.\smallbreak
    
    Notice that, yet again $h_1$ takes values in the finite set $\{0,1,\ldots,N-1\}$, and moreover, for every $q$ in $\mathcal{L}_0^{\le N-1}$, the length $h_1(q)=0$. Applying the same arguments as above one can show that every stalk $S_q$ for $q\in\mathcal{L}$ such that $h_1(q)\le N-1$ lies in $\operatorname{Loc}(\psi(T))$. We denote by $\mathcal{L}_1^{\le N-1}$ the set of all such vertices. Repeating this procedure, in each step we define a length function $h_\kappa\colon\mathcal{L}\longrightarrow\mathbb{Z}_{\ge0}$ by
    \[
    h_\kappa(q)=\operatorname{max}\{n\ge0\,|\, q=q_0\to q_1\to q_2\to\ldots\to q_n,\text{ with }q_n\in\mathcal{L}_{\kappa-1}^{\le N-1}\}
    \]
    where $q=q_0\to q_1\to q_2\to\ldots\to q_n$ is a finite chain of non-zero pseudo-radical morphisms in $\mathcal{L}$ and
  \[
  \mathcal{L}_{\kappa-1}^{\le N-1}:=\{q\in\mathcal{L}\,|\, h_{\kappa-1}(q)=0\},
  \]
  proving that every stalk of $\mathcal{L}_\kappa^{\le N-1}$ belongs to $\operatorname{Loc}(\psi(T))$. Finally,  notice that since $Q$ and thus also $\mathcal{L}$, are acyclic, locally bounded with a nilpotent pseudo-radical and $\mathcal{L}$ is closed under predecessors (see Definition \ref{dfn:L}(1)), it is straightforward to check that every stalk in $\mathcal{L}$ lies in $\operatorname{Loc}(\psi(T))$.\smallbreak

  \noindent\textbf{\textsf{Stalks in} $\mathcal{R}$:} We already shown that every finitely presented $\Lambda$-module lies in $\operatorname{Loc}(\psi(T))$ and thus so do the images of the indecomposable injective $\Lambda$-modules. Therefore, one can use the dual arguments by working with the injectives and the dual version of \prpref{filtration} (that is, every representation supported in $\mathcal{R}$ is co-filtered by stalks in $\mathcal{R}$), to get the desired result. We point out that, $\operatorname{Loc}(\psi(T))$ is \emph{not} closed under limits however, as in the case of $\mathcal{L}$, the cokernel in the dual of the short exact sequence \eqref{f-ses} is going to admit only a finite filtration. \smallbreak

  We conclude that $\operatorname{Loc}(\psi(T))$ contains a set of compact generators, thus it equals $\QSD{Q}{\Bbbk}$, proving that $\psi(T)$ generates, and thus it is a silting object.
\end{proof}
}

As explained above, every silting object gives rise to a silting t-structure. In particular, given a tilting $\Lambda$-module $T$ and denoting its image under $\psi$ by $\check{T}$, one has a silting t-structure $(\check{T}^{\perp_{>0}},\check{T}^{\perp_{\le0}})$ in $\QSD{Q}{\Bbbk}$ with heart $\mathscr{H}_{\check{T}}$, which by \cite[Prop.~4.3]{psaroudakis2018realisation} admits a projective generator $H^0_{\check{T}}(\check{T})$. Since moreover $\check{T}$ is a compact object, the heart $\mathscr{H}_{\check{T}}$ is equivalent to the module category of the endomorphism ring of $\check{T}$ in $\QSD{Q}{\Bbbk}$, that is
    \[
    \mathscr{H}_{\check{T}}\,\cong\, \Mod\text{-}\operatorname{End}_{\QSD{Q}{\Bbbk}}(\check{T}).
    \]
Furthermore, there exists a functor $\mathbf{D}(\mathscr{H}_{\check{T}})\longrightarrow\QSD{Q}{\Bbbk}$, which is an equivalence if and only if $\check{T}$ is tilting by a result of Virili \cite[Thm.~7.7]{virili2018morita}.\smallbreak

In Section \ref{sec:t-structures}, we constructed a t-structure induced by an admissible partition of $Q$. The following result establishes that these two t-structures in fact coincide for special silting objects. 

\begin{prp}\label{prp:equality_of_tstructures}
    Let $T$ be  a small projective generator of $\Lambda$ and denote by $\check{T}$ the image of $T$ in $\QSD{Q}{\Bbbk}$. Then the silting t-structure $(\check{T}^{\perp_{>0}},\check{T}^{\perp_{\le0}})$ in $\QSD{Q}{\Bbbk}$ induced by $\check{T}$, coincides with the suspended t-structure $\Sigma\,\mathbf{t}=(\underline{\mathcal{X}}_\mathcal{L},\Sigma\underline{\mathcal{Y}_\mathcal{L}})$ and the heart $\mathscr{H}_{\check{T}}$ is equivalent to $\lMod{\Lambda}$.
\end{prp}

\begin{proof}
 To show that the t-structure $(\check{T}^{\perp_{>0}},\check{T}^{\perp_{\le0}})$ coincides with the t-structure $\Sigma\,\mathbf{t}=(\underline{\mathcal{X}}_\mathcal{L},\Sigma\underline{\mathcal{Y}_\mathcal{L}})$, we prove that the aisles coincide. In particular, notice that since $\check{T}$ is a compact silting object, by \cite[Prop.4.8]{angeleri2019silting} one has that $\check{T}^{\perp_{>0}}=\operatorname{Susp}(\check{T})$, the smallest suspended subcategory of $\QSD{Q}{\Bbbk}$ closed under coproducts and direct summands that contains $\check{T}$. It is worth to remark here that the silting objects considered in \emph{loc. cit.} are \emph{large silting objects}, however, by \cite[Ex.~4.2]{psaroudakis2018realisation} every silting object in the sense of Aihara-Iyama (which are the silting objects that we consider here) can also be considered large silting objects. We thus show that $\underline{\mathcal{X}}_\mathcal{L}=\operatorname{Susp}(\check{T})$.\smallbreak
 
 One inclusion is clear, since $T$ lies in $\mathcal{X}_\mathcal{L}$ and $\underline{\mathcal{X}}_\mathcal{L}$ is an aisle (thus closed under positive suspensions and coproducts), one has $\operatorname{Susp}(\check{T})\subseteq\underline{\mathcal{X}}_\mathcal{L}$.
 For the converse inclusion, let $X$ lie in $\underline{\mathcal{X}}_\mathcal{L}$. By Proposition \ref{prp:filtration} $X$ is filtered by stalk objects in $\mathcal{L}$, thus it suffices to show that every stalk object in $\mathcal{L}$ belongs in $\operatorname{Susp}(\check{T})$, since the latter as an aisle is closed under homotopy colimts. However, this is already shown in the proof of \thmref{tilting}. More concretely, every stalk object in the boundary $\partial\mathcal{L}$ automatically lies in $\operatorname{Susp}(\check{T})$ since they can be regarded as simple $\Lambda$-modules and $T$ is a projective generator of $\Lambda$. Then, assuming that every stalk in $\mathcal{L}^{<n}_0$ lies in $\operatorname{Susp}(\check{T})$ one can show that the kernel $K$ from \eqref{f-ses} is a finite extension of objects in $\operatorname{Susp}(\check{T})$ and hence it belongs there as well. Consequently, following similar arguments as in the proof of \thmref{tilting}, one has inductively that every stalk in $\mathcal{L}$ is given by $S_q\cong\Sigma K$ and hence it lies in $\operatorname{Susp}(\check{T})$. We conclude that $\underline{\mathcal{X}}_\mathcal{L}\subseteq\operatorname{Susp}(\check{T})$.  \smallbreak

 Finally, since $T$ is a projective generator of $\Lambda$, that is, via Morita's Theorem \cite{KMr58} one computes that 
\[
    \mathscr{H}_{\check{T}}\,\cong\, \Mod\text{-}\operatorname{End}_{\QSD{Q}{\Bbbk}}(\check{T})^{\text{op}}\,\cong\,\lMod{\Lambda}.
    \]
\end{proof}

\begin{rmk}\label{cor:tilting_change_generators}
    Throughout this section we were searching for silting objects in the $Q$-shaped derived category of the field $\Bbbk$. However our results can be extended for the $Q$-shaped derived category of any $\Bbbk$-algebra. Consider $\mathcal{T}$ to be a silting subcategory in $\QSD{Q}{\Bbbk}$, $A$ be a $\Bbbk$-algebra and $\mathcal{P}_\mathcal{T}$ be the full dg subcategory spanned by the objects in $\mathcal{T}$. Then the $Q$-shaped derived category $\QSD{Q}{A}$ of $A$, admits a silting subcategory that is isomorphic to $A\otimes H^{\le0}(\mathcal{T})$ as graded categories. In the special case that $\mathcal{T}$ is a tilting subcategory, $\QSD{Q}{A}$ admits a tilting subcategory that is isomorphic to $A\otimes \mathcal{T}$ as $\Bbbk$-linear categories and moreover, there exists a triangle equivalence
    \[
    \QSD{Q}{A}\stackrel{\cong}{\longrightarrow}\mathbf{D}(A\otimes \mathcal{T}).
    \]
 The proof of this claim follows by Proposition \ref{prp:equality_of_tstructures} and \cite[Cor.~3.30~\&~Cor.~3.31]{jasso2025qshaped}. The only point that we would like to stress is that in our case we can consider any $\Bbbk$-algebra $A$, without assuming flatness over $\Bbbk$, in contrast to the Setup in Jasso's work, since this assumption is satisfied automatically by working over a field.
\end{rmk}

\section{Examples of tilting objects in \texorpdfstring{$Q$}{Q}-shaped derived categories}\label{sec:examples}

The results of \prpref{equality_of_tstructures} and Corollary \ref{cor:tilting_change_generators} recover certain already known instances of tilting in $Q$-shaped derived categories. The first one is encountered in \cite[Prop.~4.11]{MR3742439}, where Iyama, Kato and Miyachi proved that the derived category of $N$-complexes over a ring $R$ is equivalent to the derived category of the path algebra $RA_{N-1}$. In their proof they rely heavily on the existence of a tilting subcategory which they define as the subcategory of sequences of $N-2$-morphisms (see \cite[Def.~4.1]{MR3742439}). In the following Example \ref{ex: n-complexes} we show that this subcategory coincides with $\lMod{\partial\mathcal{L}, A}$. 

\begin{exa}\label{ex: n-complexes}
    Consider $Q$ to be the following quiver 
    \[\begin{tikzcd}
	{Q=} & \cdots & \bullet & \bullet & \bullet & \bullet & {\cdots,} & {\delta^N=0}
	\arrow[from=1-2, to=1-3]
	\arrow["\delta", from=1-3, to=1-4]
	\arrow["\delta", from=1-4, to=1-5]
	\arrow["\delta", from=1-5, to=1-6]
	\arrow[from=1-6, to=1-7]
\end{tikzcd}\]
where the vertices are indexed by the integers and the composition of $N$ consecutive arrows is zero, for some $N\ge2$. It is proven in \cite{HJ-JLMS} that the $Q$-shaped derived category $\QSD{Q}{A}$ is equivalent to the derived category of $N$-complexes $\mathbf{D}_N(A)$ of $A$-modules. In this setup, consider the admissible partition $\mathcal{L},\mathcal{R}$ of $Q$ given by
\[
\mathcal{L}\,=\,\{q\in \mathbb{Z}\,|\, q\le0\} \quad\text{and}\quad \mathcal{R}\,=\,\{q\in \mathbb{Z}\,|\, q>0\}.
\]
The boundary $\partial\mathcal{L}$ of $\mathcal{L}$ is highlighted in red in the following graph:
\[\begin{tikzcd}[column sep=2.2em]
	\cdots & {\stackrel{-N}{\circ}} & {\stackrel{-(N-1)}{\circ}} & {\stackrel{-(N-2)}{{\color{red}\bullet}}} & \cdots & {\stackrel{-2}{{\color{red}\bullet}}} & {\stackrel{-1}{{\color{red}\bullet}}} & {\stackrel{0}{{\color{red}\bullet}}} & {\stackrel{1}{\circ}} & \cdots
	\arrow[from=1-1, to=1-2]
	\arrow[from=1-2, to=1-3]
	\arrow[from=1-3, to=1-4]
	\arrow[color={rgb,255:red,255;green,0;blue,0}, from=1-4, to=1-5]
	\arrow[color={rgb,255:red,255;green,0;blue,0}, from=1-5, to=1-6]
	\arrow[color={rgb,255:red,255;green,0;blue,0}, from=1-6, to=1-7]
	\arrow[color={rgb,255:red,255;green,0;blue,0}, from=1-7, to=1-8]
	\arrow[from=1-8, to=1-9]
	\arrow[from=1-9, to=1-10]
\end{tikzcd}\]

It is straightforward to check that $\lMod{\partial\mathcal{L}}$ coincides with the category of representations of the Dynkin quiver $A_{N-1}$. In this case, the tilting module of $\Bbbk A_{N-1}$ is given by $T=\bigoplus_{i=1}^{N-1}P_i$, the direct sum of the indecomposable projective representations of $A_{N-1}$. The image of this object in $\QSD{Q}{\Bbbk}$ is a silting object in $\QSD{Q}{\Bbbk}$ by \ref{thm:tilting} and one can prove that it spans precisely the same subcategory discussed in \cite[Prop.~4.7]{MR3742439}, thus we conclude that it is in fact a tilting object. Finally, one has a triangle equivalence
\[
\QSD{Q}{\Bbbk}\stackrel{\cong}{\longrightarrow}\mathbf{D}(T_{N-1}(\Bbbk)),
\]
which by Remark \ref{cor:tilting_change_generators} can by extended to an equivalence between $\QSD{Q}{A}$ and $\mathbf{D}(T_{N-1}(A))$,  where $\mathbf{D}(T_{N-1}(A))$ is the derived category of upper triangular $(N-1)\times(N-1)$ matrices with entries in $\lMod{A}$, or equivalently the derived category of the category $\operatorname{Rep}(A_{N-1},\lMod{A})$ of $\lMod{A}$-valued representations of $A_{N-1}$ (see also \cite[Prop.~4.11]{MR3742439}).\smallbreak

In the special case that $N=2$ then $\mathbf{Ch}_N(A)$ is simply the (classical) category of complexes of $A$-modules $\mathbf{Ch}(A)$. The above construction recovers the standard t-structure in the derived category $\mathbf{D}(A)$ of $A$.
\end{exa}

The second instance of tilting in $Q$-shaped derived categories is encountered in \cite{gratz2024tilting}. Here, Gratz, Holm, J{\o}rgensen and Stevenson proved that for certain categories $Q$ and certain algebras $A$, one can find a tilting object $T$ and a triangle equivalence $\mathbf{D}_Q(A)\longrightarrow\mathbf{D}(\operatorname{End}(T)\otimes A)$, where $\mathbf{D}(\operatorname{End}(T)\otimes A)$ is the derived category of the tensor product of the endomorphism ring of $T$ with $A$. This result was later generalised in the context of dg-categories by Jasso in \cite{jasso2025qshaped}. This recovers well known results in the literature, such as equivalences between derived categories of $N$-complexes and derived categories of upper triangular $(N-1)\times(N-1)$ matrices as established in \cite{MR3742439} (see also Example \ref{ex: n-complexes}), a connection to Beilinson's theorem (cf. \cite{beilinson1978coherent}) for the derived category of coherent sheaves of $\mathbb{P}^n$, see \cite[Sec.~4.2]{gratz2024tilting}, and others.

\begin{exa}\label{ex: tilt a3}
    Consider $Q$ to be given by the repetitive quiver of $A_3$ modulo the mesh relations, or schematically by the following graph
    
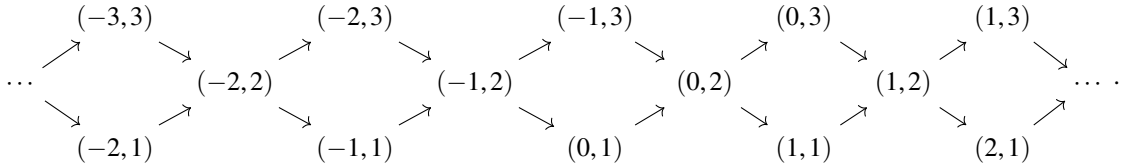
\[\begin{tikzcd}[column sep=0.7em, row sep=0.7em]
	& {(-3,3)} && {(-2,3)} && {(-1,3)} && {(0,3)} && {(1,3)} \\
	\cdots && {(-2,2)} && {(-1,2)} && {(0,2)} && {(1,2)} && \cdots \\
	& {(-2,1)} && {(-1,1)} && {(0,1)} && {(1,1)} && {(2,1)}
	\arrow[from=1-2, to=2-3]
	\arrow[from=1-4, to=2-5]
	\arrow[from=1-6, to=2-7]
	\arrow[from=1-8, to=2-9]
	\arrow[from=1-10, to=2-11]
	\arrow[from=2-1, to=1-2]
	\arrow[from=2-1, to=3-2]
	\arrow[from=2-3, to=1-4]
	\arrow[from=2-3, to=3-4]
	\arrow[from=2-5, to=1-6]
	\arrow[from=2-5, to=3-6]
	\arrow[from=2-7, to=1-8]
	\arrow[from=2-7, to=3-8]
	\arrow[from=2-9, to=1-10]
	\arrow[from=2-9, to=3-10]
	\arrow[from=3-2, to=2-3]
	\arrow[from=3-4, to=2-5]
	\arrow[from=3-6, to=2-7]
	\arrow[from=3-8, to=2-9]
	\arrow[from=3-10, to=2-11]
\end{tikzcd}.\]

We define an admissible partition $\mathcal{L},\mathcal{R}$ of $Q$ given by
\[
\mathcal{L}\,=\,\{(i,j)\in Q\,|\, i<0\} \quad\text{and}\quad \mathcal{R}\,=\,\{(i,j)\in Q\,|\, i\ge0\}.
\]
In the following graph we distinguish $\mathcal{L}$ within the striped area, $\mathcal{R}$ within the dotted area, and we highlight the boundary $\partial\mathcal{L}$ of $\mathcal{L}$ with red.\smallbreak

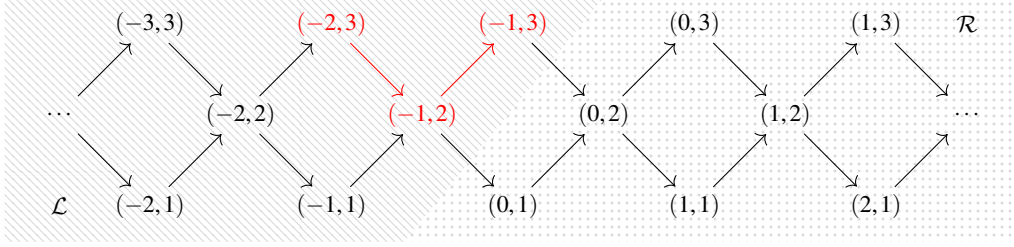
\begin{center}
\begin{tikzpicture}[
  x=1.2cm, y=1.2cm,
  every node/.style={circle,inner sep=1.2pt,font=\small},
  arr/.style={->,line width=0.4pt},
]

\usetikzlibrary{patterns}

% --- striped region (x<0) ---
\fill[pattern=north west lines,pattern color=black!40,opacity=0.6]
  (-4.6,0.6) -- (-0.2,0.6) -- (2,3.3) -- (-4.6,3.3) -- cycle;

% --- dotted region (x>=0) ---
\fill[pattern=dots,pattern color=black!40,opacity=0.6]
  (-0.2,0.6) -- (6.6, 0.6) -- (6.6,3.3) -- (2,3.3) -- cycle;

% --- nodes ---
\node (-4,3) at (-4,3) {$\,$};
\node (-4,2) at (-4,2) {$\mathbf{\cdots}$};
\node (-4,1) at (-4,1) {$\mathbf{\mathcal{L}}$};

\node (-3,3) at (-3,3) {$(-3,3)$};
\node (-3,2) at (-3,2) {$\,$};
\node (-3,1) at (-3,1) {$(-2,1)$};

\node (-2,3) at (-2,3) {$\,$};
\node (-2,2) at (-2,2) {$(-2,2)$};
\node (-2,1) at (-2,1) {$\,$};

\node (-1,3) at (-1,3) {{\color{red}$(-2,3)$}};
\node (-1,2) at (-1,2) {$\,$};
\node (-1,1) at (-1,1) {$(-1,1)$};

\node (0,3) at (0,3) {$\,$};
\node (0,2) at (0,2) {{\color{red}$(-1,2)$}};
\node (0,1) at (0,1) {$\,$};

\node (1,3) at (1,3) {{\color{red}$(-1,3)$}};
\node (1,2) at (1,2) {$\,$};
\node (1,1) at (1,1) {$(0,1)$};

\node (2,3) at (2,3) {$\,$};
\node (2,2) at (2,2) {$(0,2)$};
\node (2,1) at (2,1) {$\,$};

\node (3,3) at (3,3) {$(0,3)$};
\node (3,2) at (3,2) {$\,$};
\node (3,1) at (3,1) {$(1,1)$};

\node (4,3) at (4,3) {$\,$};
\node (4,2) at (4,2) {$(1,2)$};
\node (4,1) at (4,1) {$\,$};

\node (5,3) at (5,3) {$(1,3)$};
\node (5,2) at (5,2) {$\,$};
\node (5,1) at (5,1) {$(2,1)$};

\node (6,3) at (6,3) {$\mathbf{\mathcal{R}}$};
\node (6,2) at (6,2) {$\mathbf{\cdots}$};
\node (6,1) at (6,1) {$\,$};

% --- arrows (same as tikzcd) ---

\draw[shorten <=10pt,shorten >=10pt,->] (-4,2) -- (-3,3);
\draw[shorten <=10pt,shorten >=10pt,->] (-4,2) -- (-3,1);

\draw[shorten <=10pt,shorten >=10pt,->] (-3,3) -- (-2,2);
\draw[shorten <=10pt,shorten >=10pt,->] (-3,1) -- (-2,2);

\draw[shorten <=10pt,shorten >=10pt,->] (-2,2) -- (-1,3);
\draw[shorten <=10pt,shorten >=10pt,->] (-2,2) -- (-1,1);

\draw[shorten <=10pt,shorten >=10pt,->,draw=red] (-1,3) -- (0,2);
\draw[shorten <=10pt,shorten >=10pt,->] (-1,1) -- (0,2);

\draw[shorten <=10pt,shorten >=10pt,->,draw=red] (0,2) -- (1,3);
\draw[shorten <=10pt,shorten >=10pt,->] (0,2) -- (1,1);

\draw[shorten <=10pt,shorten >=10pt,->] (1,3) -- (2,2);
\draw[shorten <=10pt,shorten >=10pt,->] (1,1) -- (2,2);

\draw[shorten <=10pt,shorten >=10pt,->] (2,2) -- (3,3);
\draw[shorten <=10pt,shorten >=10pt,->] (2,2) -- (3,1);

\draw[shorten <=10pt,shorten >=10pt,->] (3,3) -- (4,2);
\draw[shorten <=10pt,shorten >=10pt,->] (3,1) -- (4,2);

\draw[shorten <=10pt,shorten >=10pt,->] (4,2) -- (5,3);
\draw[shorten <=10pt,shorten >=10pt,->] (4,2) -- (5,1);

\draw[shorten <=10pt,shorten >=10pt,->] (5,3) -- (6,2);
\draw[shorten <=10pt,shorten >=10pt,->] (5,1) -- (6,2);

\end{tikzpicture}
\end{center}

Here it is straight forward to check that $\lMod{\partial\mathcal{L}}$ coincides with the category of representations of the quiver \[\begin{tikzcd}
	{\Delta=} & \bullet & \bullet & {\bullet,} & {ba=0}
	\arrow["a"', from=1-2, to=1-3]
	\arrow["b"', from=1-3, to=1-4]
\end{tikzcd}.\]
We choose as a tilting module of $\Bbbk\Delta$ the direct sum of the indecomposable projective representations, that is:
\[\begin{tikzcd}[column sep=1.2em, row sep=1em]
	{T=} & {\Big(0} & 0 & {\Bbbk\Big)} & \oplus & {\Big(0} & \Bbbk & {\Bbbk\Big)} & \oplus & {\Big(\Bbbk} & \Bbbk & {0\Big)}
	\arrow[from=1-2, to=1-3]
	\arrow[from=1-3, to=1-4]
	\arrow[from=1-6, to=1-7]
	\arrow["1", from=1-7, to=1-8]
	\arrow["1", from=1-10, to=1-11]
	\arrow[from=1-11, to=1-12]
\end{tikzcd}.\]
In this case one can show that the image of $T$ in $\QSD{Q}{\Bbbk}$ (which we know by Theorem \ref{thm:tilting} that is a silting object) is in fact a tilting object by explicitly computing $$\Hom{\QSD{Q}{\Bbbk}}{\psi(T)}{\Sigma^{-n}\psi(T)}$$ for all $n>0$. In particular, considering the Frobenius model of the category of cofibrant objects whose stable category is equivalent to $\QSD{Q}{\Bbbk}$, and under the identification $\Sigma^{-1}\,\cong\,\Omega$, one can even prove that there are no morphisms from $T$ to $\Omega T$ even in $\lMod{Q}$ and hence also from $T$ to $\Omega^nT$ for every $n>0$. Thus, $\psi(T)$ is a tilting object in $\QSD{Q}{\Bbbk}$ and there is an equivalence $\QSD{Q}{\Bbbk}\,\cong\, \mathbf{D}(\operatorname{Rep}(\Delta))\,\cong\,\mathbf{D}(\Bbbk A_3)$ of triangulated categories, where the second equivalence is a well-known instance of derived Morita equivalence. These computations are straightforward and left to the reader. Finally, for an arbitrary $\Bbbk$-algebra $A$, by Remark \ref{cor:tilting_change_generators} one has a triangle equivalence,
\[
\QSD{Q}{A}\stackrel{\cong}{\longrightarrow}\mathbf{D}(\operatorname{Rep}(\Delta,\lMod{A})\stackrel{\cong}{\longrightarrow}\mathbf{D}(A\otimes\Bbbk A_3),
\]
which recovers the already known result in \cite{gratz2024tilting}.
\end{exa}

So far we showed that for any $Q$ which admits an admissible partition in the sense of Definition \ref{dfn:L}, one can construct a silting (in fact in some case even a tilting) object in $\QSD{Q}{A}$ based completely on the combinatorics of $\mathcal{L}$. Consequently, in these cases, the $Q$-shaped derived category is equivalent to the derived category of some ring. Combining our results, with some already known instances of tilting in $Q$-shaped derived categories in the literature (cf. \cite{MR3742439,gratz2024tilting,jasso2025qshaped}), naturally arises the question, whether every $Q$-shaped derived category admits silting objects. As we can see in the following subsection, the answer to this question is negative.

\subsection{Non-Example: The cyclic quiver}
Consider the cyclic quiver $Q^m$ given by a set of $m$ vertices and $m$-arrows between them such as any two consecutive arrows compose to zero (see Figure \ref{fig:cyclic_quiver}). By abuse of notation we denote by $Q^m$ also the path category modulo the ideal of relations in $Q^m$. 

\begin{figure}[h]
%\begin{center}
\begin{tikzpicture}[scale=2]
  \node at (0:1.0){$0$};
  \draw[->] (10:1.0) arc (10:46:1.0);
  \node at (60:1.0){$m-1$};
  \draw[->] (74:1.0) arc (74:102:1.0);
  \node at (120:1.0){$m-2$};
  \draw[->] (133:1.0) arc (133:168:1.0);
  \node at (175:1.0){$\cdot$};
  \node at (180:1.0){$\cdot$};
  \node at (185:1.0){$\cdot$};  
  \draw[->] (192:1.0) arc (192:232:1.0);
  \node at (240:1.0){$2$};
  \draw[->] (248:1.0) arc (248:293:1.0);
  \node at (300:1.0){$1$};
  \draw[->] (308:1.0) arc (308:352:1.0);
\end{tikzpicture}
%\end{center}
\caption{Cyclic quiver with $m$ vertices and relations: every two consecutive arrows compose to zero.}
\label{fig:cyclic_quiver}
\end{figure}
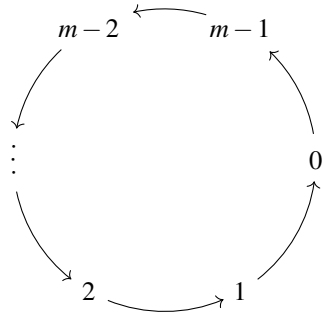
%\newpage
\begin{prp}\label{prp: no t-structures}
    Let $Q^m$ be the cyclic quiver in Figure \ref{fig:cyclic_quiver}. Then, the $Q^m$-shaped category $\mathbf{D}_{Q^m}(\Bbbk)$ admits only stable t-structures.
\end{prp}

\begin{proof}
    We begin by observing that in this case the bound quiver algebra $\Bbbk Q^m/I$, where $I$ is the ideal of relations described in Figure \ref{fig:cyclic_quiver}, is a representation finite self-injective artin algebra. Consider also the subcategory ${}_{Q^m}\fpmod:=\operatorname{Fun}(Q^m,{}_\Bbbk\fpmod)$ of $\lMod{Q^m}$ consisting of all the $\Bbbk$-linear functors from $Q^m$ that take values in the finitely presented $\Bbbk$-modules. Notice that this subcategory is precisely the subcategory of finitely presented objects of $\lMod{Q^m}$. Moreover, ${}_{Q^m}\fpmod$ is a Krull-Shcmidt category whose only non projective-injective indecomposable modules are the simples. Finally, the $Q^m$-shaped derived category in this case, is equivalent to the stable module category ${}_{Q^m}\StMod$.\smallbreak

    Assume now that $(\mathcal{C},\mathcal{D})$ is a cotorsion pair in ${}_{Q^m}\fpmod$ such that $\mathcal{C}$ is closed under co-syzygies. Then $(\mathcal{C},\mathcal{D})$ is either the projective cotorsion pair $({}_{Q^m}\operatorname{prj},{}_{Q^m}\fpmod)$ or the injective cotorsion pair $({}_{Q^m}\fpmod, {}_{Q^m}\operatorname{prj})$. This follows from the fact that if $\mathcal{C}$ contains at least one simple representation, then it contains all of them since they all appear in the same totally acyclic projective resolution
    %\small
\[\begin{tikzcd}
	\cdots & {P_{m-1}} && {P_0} && {P_1} && {P_2} & \cdots \\
	&& {S_{m-1}} && {S_0} && {S_1}
	\arrow[from=1-1, to=1-2]
	\arrow[from=1-2, to=1-4]
	\arrow[two heads, from=1-2, to=2-3]
	\arrow[from=1-4, to=1-6]
	\arrow[two heads, from=1-4, to=2-5]
	\arrow[from=1-6, to=1-8]
	\arrow[two heads, from=1-6, to=2-7]
	\arrow[from=1-8, to=1-9]
	\arrow[tail, from=2-3, to=1-4]
	\arrow[tail, from=2-5, to=1-6]
	\arrow[tail, from=2-7, to=1-8]
\end{tikzcd}\]
%\normalsize
\smallbreak

Therefore, from \cite{SaorinStovicek} ${}_{Q^m}\stmod$ has no non-trivial t-structures. By a classical result in representation theory of artin algebras by Ringel and Tachikawa \cite{ringel1975qf}, since ${}_{Q^m}\fpmod$ is representation finite, every representation in $\lMod{Q^m}$ can be written as a direct sum of finitely presented indecomposable $Q^m$-representations. Thus, every complete cotorsion pair in $\lMod{Q^m}$ is of finite type and in particular, it restrict to a complete cotorsion pair in ${}_{Q^m}\fpmod$. It follows that the cotorsion pairs in $\lMod{Q^m}$ are controlled by the cotorsion pairs in the category of finitely presented modules and since in ${}_{Q^m}\fpmod$ the only cotorsion pairs that lead to t-structures via Saor{\'\i}n-{\v{S}}t'ov{\'\i}{\v{c}}ek's correspondence Theorem \ref{thm: soarin-stovicek correspondence} are the trivial ones, we conclude that $\QSD{Q^m}{\Bbbk}$ admits only stable t-structures, whose hearts are trivial and thus  $\QSD{Q^m}{\Bbbk}$ does not have any silting objects.
\end{proof}

\begin{rmk}
    The result of Proposition \ref{prp: no t-structures} is not surprising since the triangulated category ${}_{Q^m}\StMod$ is $m$-periodic, in the sense that $\Sigma^m\cong\operatorname{id}$. It is known that periodic triangulated categories contain no tilting objects. In the special case where we consider the path algebra of the quiver $Q^m$ modulo $\operatorname{rad}^m$, it is shown in \cite[Ex.~5.8]{saito2023tilting} that $\QSD{Q^m}{\Bbbk}$ is triangle equivalent to the 2-periodic derived category $\mathbf{D}^{2\text{-per}}(\Bbbk A_{m-1})$ of the path algebra of the quiver $A_{m-1}$.
\end{rmk}

\bibliographystyle{alpha}
\bibliography{refs}

\end{document}